\documentclass[10pt,reqno]{amsart}
\usepackage{amsmath}
\usepackage{amsfonts}
\usepackage{mathrsfs}
\usepackage{amssymb}
\usepackage{xypic}
\usepackage{amsthm}
\usepackage{url}
\usepackage{setspace}     \spacing{1}
\usepackage{marginnote}
\usepackage{enumitem}
 \usepackage[usenames,dvipsnames]{pstricks}
 \usepackage{epsfig}
 \usepackage{pst-grad} 
 \usepackage{pst-plot} 
 
\usepackage{xargs}
\usepackage[colorlinks=true]{hyperref} \hypersetup{urlcolor=blue, citecolor=red}
\usepackage{times}

\theoremstyle{plain}
\newtheorem{theorem}{Theorem}[section]
\newtheorem{lemma}[theorem]{Lemma}
\newtheorem{proposition}[theorem]{Proposition}
\newtheorem{prop}[theorem]{Proposition}
\newtheorem{corollary}[theorem]{Corollary}

\theoremstyle{definition}
\newtheorem{definition}[theorem]{Definition}
\newtheorem{remark}[theorem]{Remark}

\usepackage{mathabx}

\newcommand{\diff}{\mathrm{Diff}}

\newcommand{\Sl}{\mathrm{SL}}
\newcommand{\SL}{\mathrm{SL}}

\newcommand{\id}{\mathrm{Id}}

\title[Measure rigidity for random dynamics on surfaces]{Measure rigidity for  random dynamics on surfaces with positive entropy} 
\author[A.~W.~Brown]{Aaron W. Brown}
\address{Penn State University}
\email{brown@math.psu.eu}

\author[F.~Rodriguez~Hertz]{Federico Rodriguez Hertz}
   \address{Penn State University}
\email{hertz@math.psu.edu}

\long\def\symbolfootnote[#1]#2{\begingroup\def\thefootnote{\fnsymbol{footnote}}
\footnote[#1]{#2}\endgroup}

\def\I{\mathcal I}
\def\Exp{\mathbb E}
\def\Ex{\Exp}
\def\ae{a.e.\ }
\def\as{a.s.\ }

\def\U{\mathcal U}

\newcounter{step}
\newcommand{\step}{\addtocounter{step}{1}\arabic{step}}
\def\lip{\mathrm{Lip}}

\begin{document}
\newcommand\Sigmaloc{\Sigma^{\pm}_{\text {loc}}}
\newcommand\Sigmalocu{\Sigma^{+}_{\text {loc}}}
\newcommand\Sigmalocs{\Sigma^{-}_{\text {loc}}}
\newcommand{\Fols}{\mathcal{W}^s}
\newcommand{\Folu}{\mathcal{W}^u}
\renewcommand{\L}{\mathcal{L}}
\newcommand{\Fol}{\mathcal{F}}

\newcommand{\aeq}{\circeq}

\newcommand{\cmt}[1]{{\color{red}{{#1}}}}

 \newcommand{\oldepsilon}{\mathchar"10F}
\renewcommand{\epsilon}{\varepsilon}
\renewcommand{\emptyset}{\varnothing}
\newcommand{\restrict}[2]{{#1}{\restriction_{{ #2}}}}
\newcommand{\restrictThm}[2]{{#1}\! \!  \restriction_{\! #2}}

\newcommand{\sm}{\smallsetminus}
\newcommand{\R}{\mathbb {R}}
\newcommand{\T}{\mathbb {T}}

\newcommand{\Q}{\mathcal {Q}}
\newcommand{\Z}{\mathbb {Z}}
\newcommand{\N}{\mathbb {N}}
\newcommand{\A}{\mathfrak {A}}
\newcommand{\B}{\mathcal  {B}}
\renewcommand{\P}{\mathcal{P}}

\newcommand{\inv}{^{-1}}
\newcommand{\C}{\mathcal C}

\newcommand{\note}[1]{\marginnote{{\color{red}\footnotesize \begin{spacing}{1}#1\end{spacing}}}}

\def\hol{\mathcal H}

\def\us{{u/s}}
\def\td {\tilde}

\newlength{\wideitemsep}
\setlength{\wideitemsep}{.5\itemsep}
\addtolength{\wideitemsep}{7pt}
\let\olditem\item
\renewcommand{\item}{\setlength{\itemsep}{\wideitemsep}\olditem}

\def\A{\mathcal A}
\def\F{\mathcal F}
\def\G{\mathcal G}	
	  
\def\E{\mathbb E}
	  
\def\MP{\mathcal{X}}
\renewcommand{\underbar}{\underline}
\renewcommand{\bar}{\overline}
\def\good{\mathscr G} 

\newcommand{\stab}[2]{W^s\!\left (#2, {#1}\right)}
\newcommand{\unst}[2]{W^u\!\left(#2,{#1}\right)}
\newcommand\locstab[3][r]{W^s_{ #1}\!\left( {#3},#2 \right)}
\newcommand\locunst[3][r]{W^u_{#1}\!\left({#3}, #2 \right)}

\newcommand{\stabM}[2]{W^s_#2\!\left({#1}\right)}
\newcommand{\unstM}[2]{W^u_#2\!\left({#1}\right)}
\newcommand\locstabM[3][r]{W^s_{#3,#1}\!\left(#2\right)}
\newcommand\locunstM[3][r]{W^u_{#3, #1}\!\left(#2\right)}

\newcommand{\stabp}[1]{W^s\!\left({#1}\right)}
\newcommand{\unstp}[1]{W^u\!\left({#1}\right)}
\newcommand\locstabp[2][r]{W^s_{#1}\!\left( #2\right)}
\newcommand\locunstp[2][r]{W^u_{#1}\!\left( #2\right)}

\newcommand{\Eu}[2]{E^u\left({#2, #1}\right)}
\newcommand{\Es}[2]{E^s\left({#2, #1}\right)}
\newcommand{\EuM}[2]{E^u_{#2}\left({ #1}\right)}
\newcommand{\EsM}[2]{E^u_{#2}\left({ #1}\right)}
\newcommand{\Eup}[1]{E^u\left({ #1}\right)}
\newcommand{\Esp}[1]{E^u\left({ #1}\right)}
\renewcommand{\E}{\mathcal E}

\def\W{\mathcal W}
\def\scrF{\mathscr{F}}

\newcommand{\cocycle}[1][\xi]{%
  \def\ArgI{{#1}}%
  \BlahRelay
}
			\newcommand\BlahRelay[1][n]{%
			 %
			  f_{\ArgI} ^{#1}
			}

\maketitle
\begin{abstract}
Given a surface $M$ and a Borel probability measure $\nu$ on the group of $C^2$-diffeomorphisms of $M$, we study $\nu$-stationary probability measures on $M$.  Assuming the positivity of a certain entropy, the following dichotomy is proved: either the stable distributions for the random dynamics is non-random, or the measure is SRB.  
In the case that $\nu$-\ae diffeomorphism preserves a common smooth measure $m$, we show that for any positive-entropy stationary  measure $\mu$ either there exists a $\nu$-almost-surely invariant $\mu$-measurable line field (corresponding do the stable distributions for \ae random composition) or the measure $\mu$ is $\nu$-almost-surely  invariant and coincides with an ergodic component of $m$. 

To prove the above result, we introduce a  skew product with surface fibers over a measure preserving transformation equipped with an increasing sub-$\sigma$-algebra $\hat \Fol$.  Given an invariant measure $\mu$ for the skew product, and assuming the $\hat \Fol$-measurability of the `past dynamics' and the fiber-wise conditional measures, we prove a dichotomy: either the fiber-wise stable distributions are measurable with respect to a related increasing sub-$\sigma$-algebra, or the measure $\mu$ is fiber-wise SRB. 

\end{abstract}

\section{Introduction}

Given an action of a one-parameter group on a manifold with some degree of hyperbolicity, there are typically many ergodic,
invariant measures with positive entropy. 
For instance, given an Anosov or Axiom A diffeomorphism of a compact manifold, the equilibrium states for H\"older-continuous potentials provide measures with the above properties \cite{MR0380889,MR0442989}.  
When passing to hyperbolic actions of larger groups, the following phenomenon has been demonstrated in many settings: the only invariant ergodic measures with positive entropy are absolutely continuous (with respect to the ambient Riemannian volume).  
For instance, consider the action of the semi-group $\N^2$ on the additive circle generated by
$$x\mapsto 2x \bmod 1\quad \quad x \mapsto 3x \bmod 1.$$
Rudolph showed for this action that the only invariant, ergodic probability measures are Lebesgue or have zero-entropy for every one-parameter subgroup \cite{MR1062766}.  
In \cite{MR1406432}, Katok and Spatzier generalized the above phenomenon to  actions of commuting toral automorphisms.  

Outside of the setting of affine actions,  Kalinin,  Katok, and Rodriguez Hertz, have recently demonstrated a version of abelian measure rigidity for nonuniformly hyperbolic, maximal-rank actions. 
In  \cite{MR2811602}, the authors consider $\Z^n$ acting by $C^{1+\alpha}$ diffeomorphisms on a $(n+1)$-dimensional manifold and prove that any  $\Z^n$-invariant measure $\mu$ is absolutely continuous assuming that 
at least one element of $\Z^n$ has positive entropy with respect to $\mu$ and the Lyapunov exponent functionals are in \emph{general position}.  

For affine actions of non-abelian groups, a number of results have recently been obtained by  Benoist and Quint  in a series of papers  \cite{MR2831114,MR3037785,BQIII}.  For instance, consider a finitely supported measure $\nu$ on the group $\SL(n,\Z)$.  Let $\Gamma_\nu\subset \Sl(n,\Z)$ be the { (semi-)}group generated by the support of $\mu$.  We note that $\Gamma_\nu$ acts naturally on the torus $\T^n$.  In \cite{MR2831114}, it is proved that if every finite index subgroup of { (the group generated by)} $\Gamma_\nu$ acts irreducibly on $\R^n$ then every $\nu$-stationary probability measure on $\T^n$ is either finitely supported or is Haar; in particular every $\nu$-stationary probability measure is $\SL(n,\Z)$-invariant.  Similar results are obtained  in \cite{MR2831114} for groups of translations on quotients of simple Lie groups.

In this article, we prove a measure rigidity result  for stationary measures for groups acting by diffeomorphisms on surfaces.  We focus here only on actions on surfaces  and stationary probability measures with positive entropy though we expect the results to hold in more generality.   We rely heavily on the tools from the theory of nonuniformly hyperbolic diffeomorphisms used in \cite{MR2811602} as well as a modified version of the ``exponential drift'' arguments developed in \cite{MR2831114} and \cite{1302.3320}.

\def\nunaught{{\hat \nu}}
\def\munaught{{\hat \mu}}
\section{Preliminary definitions and constructions}
Let $M$ be a closed (compact, boundaryless) $C^\infty$ Riemannian manifold.  We write  $\diff^r(M)$ for the group of $C^r$-diffeomorphisms from $M$ to itself equipped with its natural $C^r$-topology.  
Fix $r= 2$ and consider a subgroup $\Gamma\subset \diff^2(M)$. 
We say a Borel probability measure $\mu$ on $M$ is \emph{$\Gamma$-invariant} if \begin{equation}\label{eq:inv} \mu(f\inv (A)) = \mu(A)\end{equation} for all Borel $A\subset M$ and all $f\in \Gamma$.

We note that for any continuous action by an amenable group on a compact metric space, there always exists at least one invariant measure.  However, for  actions by non-amenable groups,  invariant measures need not exists.  
For this reason, we introduce   a weaker notion of invariance.  
Let $\nu$ be a Borel probability measure \emph{on the group} $\Gamma$.  
We say Borel probability measure $\mu $ on $M$ is \emph{$\nu$-stationary} if 
	$$\int \mu (f\inv (A)) \ d \nu(f)= \mu (A)$$
for any Borel $A\subset M$.  	
By the compactness of $M$, it follows that  for any probability $\nu$ on $\Gamma$ there exists a $\nu$-stationary probability $\mu$ (e.g.\ \cite[Lemma I.2.2]{MR884892}.)

We note that if $\mu $ is $\Gamma$-invariant then $\mu $ is trivially $\nu$-stationary for any measure $\nu$ on  $\Gamma$.   
Given a $\nu$-stationary measure $\mu$ such that equality \eqref{eq:inv} holds for $\nu$-\ae $f\in \Gamma$, we say that $\mu$ is \emph{$\nu$-\as $\Gamma$-invariant.  }

Given a probability  $\nu$  on  $\diff^2(M)$ one defines the \emph{random walk} on group of diffeomorphisms.  A  path in the random walk induces a sequence of diffeomorphisms from $M$ to itself.  
As in the case of a single transformation, we study the asymptotic ergodic properties of typical sequences of diffeomorphisms acting on $M$.  
We write $\Sigma_+ =   \left(\diff^2(M)\right)^\N$ for the space of sequences of diffeomorphisms
	$\omega= (f_0, f_1, f_2, \dots ) \in \Sigma_+.$ 
Given a Borel probability measure $\nu $ on $\diff^2(M)$, we equip $\Sigma_+$ with the product measure $\nu ^\N$. 
We observe that $\diff^2(M)$ is a Polish space, hence $\Sigma_+$ is Polish and the probability $\nu ^\N$ is Radon.
Let $\sigma\colon \Sigma_+\to \Sigma_+$ be the shift map 
$$\sigma\colon (f_0, f_1,f_2, \dots )\mapsto (f_1,f_2, \dots).$$
We have that $\nu ^\N$ is $\sigma$-invariant. 
Given a sequence $\omega= (f_0, f_1, f_2, \dots ) \in\Sigma_+$ and $n\ge 0$ we define a cocycle
	 $$\cocycle [\omega] [0]:= \id , \quad \cocycle  [\omega][\ ] = \cocycle  [\omega][1] := f_0,\quad \cocycle [\omega] [n] := f_{n-1} \circ f_{n-2} \circ \dots \circ f_1 \circ f_0.$$
We interpret $(\Sigma_+, \nu^\Z)$ as a parametrization of all paths in the random walk defined by $\nu$.  
Following existing literature (\cite{MR968818}, \cite{MR1369243}), we denote by $\MP^+(M, \nu)$ the random dynamical system on $M$ defined by the random compositions $\{f^n_\omega\}_{\omega\in \Sigma_+}$.

Given a measure $\nu$ on $\diff^2(M)$ and a $\nu$-stationary measure $\mu$, we say a subset $A\subset M$ is \emph{$\MP^+(M, \nu)$-invariant} if for $\nu$-\ae $f$ and $\mu$-\ae $x\in M$
\begin{enumerate}
\item $x\in A \implies f(x)\in A$ and \item $x\in M\sm A \implies f(x) \in M\sm A.$\end{enumerate}
We say a $\nu$-stationary probability measure $\mu $ is \emph{ergodic} if, for every $\MP^+(M, \nu)$-invariant set $A$,
we have either $\mu (A) =0$ or $\mu (M\sm A) = 0$.  
We note that for a fixed $\nu$-stationary measure $\mu$, we have an ergodic decomposition of $\mu$ into ergodic, $\nu$-stationary measures  \cite[Proposition I.2.1]{MR884892}.

For a fixed $\nu$ and a fixed $\nu$-stationary probability $\mu $, we define the $\mu$-metric entropy of the random process $\MP^+(M, \nu )$ as follows.  Given a finite partition $\xi$ of $M$ into $\mu $-measurable sets, define $$H_\mu (\xi) = -\sum _{C_i\in \xi} \mu (C_i) \log(\mu (C_i))$$ (with the standard convention that $0\log 0 = 0$.)  
 We define the $\mu$-entropy of $\MP^+(M,\nu )$ with respect to the partition $\xi$ by 
$$h_\mu (\MP^+(M, \nu), \xi):= \lim _{n\to \infty}\dfrac 1 n \int H_\mu  \left (\bigvee_{k=0}^{n-1}(f^k_\omega)\inv \xi   \right) \ d \nu ^\N(\omega)$$ and the entropy of $\MP^+(M, \nu)$ by 
\begin{equation}\label{eq:ent}
h_\mu (\MP^+(M, \nu ):= \sup h_\mu (\MP^+(M, \nu), \xi)
\end{equation}
where the supremum is taken over all finite measurable partitions $\xi$.

	\def\ICeq{
			 &\int \log ^+(|f|_{C^2}) + \log ^+(|f\inv|_{C^2}) \ d \nu  <\infty. \tag{$\ast$}\label{eq:IC2a} 
		} 
In the case that $\nu$ is not compactly supported in $\diff^2(M)$, we will assume the integrability condition
\begin{align}\ICeq 
\end{align}
where $\log^+(a) = \max\{\log (x), 0\}$ and $|\cdot|_{C^2}$ denotes the $C^2$-norm.  
The integrability condition \eqref{eq:IC2a} implies the weaker  condition
\begin{equation}\int \log ^+(|f|_{C^1}) + \log ^+(|f\inv|_{C^1}) \ d \nu  <\infty \label{eq:IC1a}\end{equation} 
which guarantees  Oseledec's Multiplicative Ergodic theorem holds. 
The  $\log$-integrability of the $C^2$-norms is needed later to apply tools from Pesin theory. 
\begin{proposition}[Random Oseledec's multiplicative theorem.]\label{prop:OMT}
Let $\nu$ be measure on $\diff^2(M)$ satisfying \eqref{eq:IC1a}.  Let $\mu $ be an ergodic, $\nu $-stationary probability.  

Then there are real numbers $-\infty<\lambda_1<\lambda_2<\dots<\lambda_\ell<\infty$, called \emph{Lyapunov exponents} such that for $\nu ^\N$-\ae sequence $\omega\in \Sigma_+$ and $\mu $-a.e. $x\in M$ there  is a filtration
\begin{equation}\label{eq:filtration}
\{0\} \subsetneq V^1_\omega(x)\subset \dots \subsetneq V^\ell_\omega(x)= TM
\end{equation}
such that for $v\in V^k_\omega(x)\sm V^{k-1}_\omega(x)$
\[\lim _{n\to \infty} \dfrac 1 n \|D_x \cocycle  [\omega][n] (v)\| = \lambda _k.\]
Furthermore, the subspaces $V^i_\omega(x)$ are invariant in the sense that $$D_xf_\omega V^k_\omega(x) = V^k_{\sigma(\omega)}(f_\omega(x)).$$
\end{proposition}
For a proof of the above theorem see, for example, \cite[Proposition I.3.1]{MR1369243}. 
We  write $$E^s_\omega(x): = \bigcup _{\lambda_j<0}V^j_\omega(x)$$ for the \emph{stable Lyapunov subspace} for the word $\omega$ at the point $x$.  

We note that the Random process $\MP^+(M, \nu )$ is not invertible.  Thus, while stable Lyapunov subspaces are defined for $\nu^\Z$-\ae $\omega$ and $\mu$-\ae $x$, there are no well-defined unstable Lyapunov subspaces for $\MP^+(M, \nu )$.  
However, to state the result we will need a notion of SRB-measures (also called $u$-measures) for random sequences of diffeomorphisms. We will state the precise definition (Definition \ref{def:SRB}) in Section \ref{sec:condmeas} 
after introducing fiber-wise unstable manifolds for a related skew product construction.  
Roughly speaking, a $\nu$-stationary measure $\mu$ is SRB if it has absolutely continuous conditional measures along unstable manifolds.  Since we have not yet defined unstable manifolds (or subspaces), we postpone the formal definition 
and give here an equivalent property.  The following is an adaptation of \cite{MR819556}.  

\begin{prop}[{\cite[Theorem VI.1.1]{MR1369243}}]\label{prop:SRBrandom}
Let $M$ be a compact manifold and let $\nu$ be a probability on $\diff^2(M)$ satisfying \eqref{eq:IC2a}.
Then, an ergodic, $\nu $-stationary probability $\mu $ is an SRB-measure if and only if 
\[h_\mu (\MP^+(M, \nu)) = \sum_{\lambda_i>0} m_i \lambda_i\]
where $m_i= \dim V_i- \dim V_{i-1}$ is the multiplicity of the exponent $\lambda_i$.
\end{prop}

We introduce some terminology for invariant measurable subbundles.
Given a subgroup $\Gamma\subset \diff^2(M)$, we have an induced the action of $\Gamma$ on sub-vector-bundles of the tangent bundle $TM$ via the differential.  
Consider a $\nu$ supported on $\Gamma$ and a $\nu$-stationary Borel probability $\mu $ on $M$.  
\begin{enumerate}
\item 
We say a $\mu$-measurable subbundle $V\subset TM$ is  \emph{$\nu$-\as invariant} if \(Df (V(x)) = V(f(x))\)
 for every $\nu$-\ae $f\in \Gamma$ and  $\mu $-\ae $x\in M.$

\item A $(\nu^\N \times \mu) $-measurable family of subbundles $(\omega,x) \mapsto V_\omega(x)\subset T_xM$ is \emph{$\MP^+(\Gamma,\nu )$-invariant} if  for $(\nu ^\N\times \mu )$-\ae $(\omega,x)$
			\[D_xf_\omega V_\omega(x)  = V_{\sigma(\omega)}(f_\omega(x)).\]
Note that subbundles in the filtration \eqref{eq:filtration} are $\MP^+(\Gamma,\nu )$-invariant.

\item  We say a $\MP^+(M,\nu )$-invariant family of subspaces $V_\omega(x)\subset TM$ is \emph{non-random} if  there exists a  $\nu$-\as invariant $\mu $-measurable subbundle $\hat V\subset TM$ with $\hat V(x)= V_\omega(x)$  for $(\nu ^\N \times \mu)$-\ae $(\omega,x)$.
\end{enumerate}

\section{Statement of results: groups of surface diffeomorphisms} 
For all results in this paper, we restrict ourselves to the case that $M$ is a closed surface.  
Equip  $M$ with a background Riemannian metric. 

\label{sec:resultsStationary}
Let $\nunaught$ be a  Borel probability on the group $\diff^2(M)$ satisfying the integrability hypotheses \eqref{eq:IC2a}. 
Let $\munaught$ be an ergodic $\nunaught$-stationary measure on $M$.  
We will assume positivity of the metric entropy: \begin{equation}\label{eq:poz} h_\munaught(\MP^+(M, \nunaught))>0.\end{equation} 
By the fiber-wise Margulis--Ruelle inequality (see Proposition \ref{prop:MR} below) applied to the associated skew product (see Section \ref{sec:skewRDS}), 
\eqref{eq:poz} implies
that the Oseledec's filtration \eqref{eq:filtration} is nontrivial and the exponents satisfy \begin{align}-\infty<\lambda_1<0<\lambda_2<\infty.\label{eq:pozzexp}\end{align}
In particular, the {stable Lyapunov subspace} $E^s_\omega(x)$ corresponds to the  subspace $V^1_\omega(x)$ in \eqref{eq:filtration} and is 1-dimensional.

We now state our first theorem
\begin{theorem}\label{thm:1}
\label{thm:main}
Let $M$ be a closed $C^\infty$ surface. 
Let  $\nunaught$ be a Borel probability measure on $\diff^2(M)$ 
satisfying \eqref{eq:IC2a} and let $\munaught$ be an ergodic, $\nunaught$-stationary Borel probability measure on $M$ with $h_\munaught(\MP^+(M, \nunaught))>0.$ 
Then either
\begin{enumerate}
\item the stable distribution $E^s_\omega(x)$ is non-random, or \item $\munaught$ is SRB.  
\end{enumerate}
\end{theorem}

We have as an immediate corollary.
\begin{corollary}
Let  $\nunaught$  be as in Theorem \ref{thm:1} with $\munaught$ an ergodic, positive entropy, $\nunaught$-stationary probability measure.   Assume there are no $\nunaught$-\as invariant $\munaught$-measurable line-fields.  Then $\munaught$ is SRB.
\end{corollary}

We note that in \cite{MR2831114}, the authors prove an analogous statement.  Namely, for homogeneous actions satisfying certain hypotheses, any non-atomic stationary measure $\munaught$ is shown to be absolutely continuous along \emph{some} unstable (unipotent) direction.   
Using Ratner Theory, one concludes that the stationary measure $\munaught$ is thus the Haar measure and hence \emph{invariant} for every element of the action.  
In non-homogeneous settings, such are the one considered here and the one considered in \cite{1302.3320}, there is no analogue of Ratner Theory.  Thus, in such settings 
more structure is needed in order to promote the SRB property to absolutely continuity or almost-sure invariance of the stationary measure $\munaught$.   
The next theorem demonstrates that this promotion is possible assuming the existence of an almost-surely invariant volume.   
\begin{theorem}\label{thm:3} 
Let  $\Gamma\subset \diff^2(M)$ be a subgroup and assume $\Gamma$ preserves a probability measure $m$ equivalent to the Riemannian volume on $M$.
Let $\nunaught$ be a probability measure on $\diff^2(M)$ with $\nunaught(\Gamma) = 1$ and satisfying \eqref{eq:IC2a}.  Let $\munaught$ be an ergodic $\nunaught$-stationary Borel probability measure.  Then either 
\begin{enumerate}[label = (\arabic*)]
\item $h_\munaught(\MP^+(\Gamma,\nunaught)) = 0$, 
 \item $h_\munaught(\MP^+(\Gamma,\nunaught)) >0$  
and  the stable distribution $E^s_\omega(x)$ is non-random, or  
\item\label{case3} $\munaught$ is absolutely continuous and  is $\nunaught$-\as $\Gamma$-invariant.
\end{enumerate}

Furthermore, in conclusion \ref{case3}, we will have that $\munaught$ is (up to normalization) the restriction of $m$ to a positive volume subset. 
\end{theorem}

\section{Skew product (re)formulation of results}\label{sec:cocylereint}
We translate the above results about random products of diffeomorphisms into results about related skew products, as well as introducing a more abstract skew product setting.    This allows us to convert the dynamical properties of random, non-invertible actions, to properties of one-parameter invertible actions and to exploit tools from the theory of  nonuniform hyperbolicity. 
\subsection{Canonical skew product associated to a random dynamical system}\label{sec:skewRDS}  Let $M$ and $\hat\nu$ be as in Section \ref{sec:resultsStationary}. Consider the product space $\Sigma_+\times M$ and define the (non-invertible) skew product
  $\hat F \colon \Sigma_+\times M \to \Sigma_+ \times M$ by $$\hat F\colon (\omega, x) \mapsto (\sigma(\omega), f_\omega(x)).$$
Recall that the measure $\hat\nu^\N$ on $\Sigma_+$ is $\sigma$-invariant.
We have the following reinterpretation of $\nunaught $-stationary measures.  
\begin{proposition}\cite[Lemma I.2.3, Theorem I.2.1]{MR884892} \label{prop:charStatmeas}
For  a Borel probability measure $\munaught $ on $M$ we have that
\begin{enumerate}
\item $\munaught$ is $\nunaught$-stationary if and only if $\nunaught^\N \times \munaught$ is $\hat F$-invariant;
\item a $\nunaught$-stationary measure $\munaught$ is ergodic for $\MP^+(M, \nunaught)$ if and only if $\nunaught^\N \times \munaught$ is ergodic for $\hat F$.  
\end{enumerate}
\end{proposition}


\def\fthere{Writing the cocycle as $f^n_\xi$ is standard in the literature but is  somewhat ambiguous.   We write 
$(f_{\xi})\inv$ to indicate the diffeomorphism that is the inverse of $f_\xi\colon M\to M$.  The symbol $f_\xi\inv$ indicates $(f_{\theta\inv(\xi)})\inv$.}

We construct a canonical  invertible skew product (the natural extension)  where  tools from Pesin theory can be applied to the fiber dynamics. 
Let $\Sigma:= (\diff^r(M))^\Z$ be the space of bi-infinite sequences and equip $\Sigma$ with the product measure  
	$\nunaught^\Z.$
We again write  $\sigma\colon \Sigma\to \Sigma$ for the left shift  
	$ (\sigma(\xi))_i = \xi_{i+1}  .$
 Given $$\xi= (\dots, f_{-2}, f_{-1}, f_0, f_1,f_2, \dots ) \in \Sigma$$ define $f_\xi := f_0$ and define the 
the (invertible) skew product $F \colon \Sigma \times M\to \Sigma \times M$
by \begin{equation}\label{eq:skewdefn}F\colon (\xi, x) \mapsto (\sigma(\xi), f_\xi(x)).\end{equation}

We have the following proposition  producing the measure whose properties we will study for the remainder.  
 \begin{proposition}[{\cite[Proposition I.1.2]{MR1369243}}]
 \label{prop:mudef} Let $\munaught$ be a $\nunaught$-stationary Borel probably measure.
 There is a unique $F$-invariant Borel probability measure $\mu$ on $\Sigma \times M$
whose image under the canonical  projection $\Sigma\times M\to \Sigma_+\times M$  is  $\nunaught^\N\times \munaught$.

Furthermore, $\mu$ projects to $\nunaught^\Z$ and $\munaught$, respectively, under the canonical projections $\Sigma\times M\to \Sigma$ and $\Sigma\times M\to M$ and 
is equal to the weak-$*$ limit  \begin{equation}\label{eq:mulim}\mu = \lim_{n\to \infty} (F^n)_*(\nunaught^\Z\times \munaught).\end{equation}
\end{proposition}

Write  $\pi \colon \Sigma\times M \to \Sigma$ for the canonical projection.  We write  $h_{\mu} (F \mid\pi)$ for the conditional metric entropy of $(F, \mu)$ conditioned on the sub-$\sigma$-algebra generated by $\pi\inv$.  
We have the following equivalence.
\begin{proposition}[{\cite[Theorem II.1.4]{MR884892}, \cite[Theorem I.2.3]{MR1369243}}]
We have the equality of entropies
\label{prop:entropiessame}
$h_\munaught(\MP^+(M, \nunaught))= h_{\mu} (F \mid\pi).$
\end{proposition}
 We also note that the \emph{Abramov--Rohlin formula} 
	$$ h_{\mu}(F) =h_{\mu} (F \mid\pi) + h_{\nunaught^\Z}(\sigma)$$
holds in our setting (see for example \cite{MR1179170}.)

\subsection{Abstract skew products}\label{sec:absSkew}
We give a generalization of the setup introduced in Section \ref{sec:skewRDS}.  
Let $(\Omega, \B_\Omega, \nu)$ be a Polish probability space; that is, $\Omega$ has the topology of  a complete separable metric space, $\nu$ is a Borel probability measure, and $\B_\Omega$ is {the $\nu$-completion} of the Borel $\sigma$-algebra.   
Let $\theta\colon (\Omega, \B_\Omega, \nu)\to (\Omega, \B_\Omega, \nu)$ be an invertible, ergodic, measure-preserving transformation.  
Let $M$ be a closed $C^\infty$ manifold.  Fix a background $C^\infty$ Riemannian metric on $M$ and write $\|\cdot\|$ for the norm on the tangent bundle $TM$ and $d(\cdot, \cdot)$ for the induced distance on $M$.  We note that compactness of $M$ guarantees all metrics are equivalent, whence all dynamical objects structures defined below are independent of the choice of metric.

We  consider  a $\nu$-measurable mapping $\Omega\ni\xi\mapsto f_\xi\in \diff^2(M)$.  
\def\fthere{Writing the cocycle as $f^n_\xi$ is standard in the literature but is  somewhat ambiguous.   We write 
$(f_{\xi})\inv$ to indicate the diffeomorphism that is the inverse of $f_\xi\colon M\to M$.  The symbol $f_\xi\inv$ indicates $(f_{\theta\inv(\xi)})\inv$.}
As before, define\footnote{\fthere} a cocycle $\scrF\colon\Omega \times \Z\to \diff^r(M)$, written $\scrF \colon (\xi, n) \mapsto \cocycle$,  by 
\begin{enumerate}
	\item $\cocycle    [\xi] [0]:= \id$, $\cocycle    [\xi] [1] := f_\xi$,
	\item $\cocycle    [\xi] [n]:= f_{\theta^{n-1}(\xi)} \circ \dots \circ f_{\theta(\xi)}\circ f_\xi$ for $ n>0$, and
	\item $\cocycle    [\xi] [n] := (f_{\theta^{-n}(\xi)})\inv \circ \dots \circ (f_{\theta\inv(\xi)})\inv = (f_{\theta^{-n}\xi}^n) \inv $ for $ n<0.$
\end{enumerate}
 We will always assume the following integrability condition 
\begin{align}
&\int \log ^+(|f_\xi|_{C^2}) + \log ^+(|f\inv_\xi|_{C^2}) \ d \nu(\xi) <\infty. \tag{IC}\label{eq:IC2} 
\end{align}

Write $X:= \Omega\times M$ with canonical projection $\pi\colon X\to \Omega$.   For $\xi \in \Omega$, we will write $$M_\xi:= \{\xi\}\times M = \pi\inv (\xi)$$ for the fiber of $X$ over $\xi$.  
On $X$, we define the skew product $F\colon X \to X$  
$$F\colon(\xi,x)\mapsto(\theta(\xi), f_\xi(x)).$$ 
 
Note that $X= \Omega\times M$ has a natural Borel structure.  The main object of study for the  remainder will be $F$-invariant Borel probability measures on $X$ with marginal $\nu$.  We introduce terminology for such measures.  
 \begin{definition}
A probability  measure $\mu$ on $X$ is called 
\emph{$\scrF$-invariant} 
if it is $F$-invariant and  satisfies $$\pi_*\mu = \nu.$$  
Such a measure $\mu$ is said to be \emph{ergodic} if it is $F$-ergodic.  
\end{definition}

\subsubsection{Fiber-wise Lyapunov exponents}
 We define $TX$ to be the fiber-wise tangent bundle 
 $$TX:= \Omega\times TM$$ and $DF\colon TX\to TX$ to be the fiber-wise differential
 $$DF \colon (\xi, (x,v)) \mapsto (\theta(\xi), (f_\xi (x), D_xf_\xi v)).$$

Let $\mu$ be an ergodic, $\scrF$-invariant probability.  We have that $DF$ defines a linear cocycle over the (invertible) measure preserving system $F\colon (X, \mu)\to (X, \mu)$.  
By the integrability condition \eqref{eq:IC2}, we can apply Oseledec's Theorem to $DF$ to obtain a $\mu$-measurable splitting
\begin{equation}\label{eq:OscSplitting}T_{(\xi,x)}X:= \{\xi\}\times T_xM  = \bigoplus_j E^j(\xi,x)\end{equation} and numbers $\lambda_\mu^j$ so that  
for $\mu$-a.e. $(\xi,x)$, and every $v\in E^j(\xi,x)\sm\{0\}$
$$\lim_{n\to \pm \infty} \dfrac 1 n \log  \|DF^n( v)\|= \lim_{n\to \pm \infty} \dfrac 1 n \log  \|D_x \cocycle    [\xi] [n] v\| = \lambda^j_\mu.$$

 \subsubsection{Fiber entropy}

Recalling the  canonical projection   $\pi \colon X\to \Omega$, 
given an $\scrF$-invariant probability measure $\mu$ we write  $h_{\mu} (F \mid\pi)$ for the conditional metric entropy of $(F, \mu)$ conditioned on the sub-$\sigma$-algebra generated by $\pi\inv$.   
We have the following generalized Margulis--Ruelle inequality.
\begin{proposition}[\cite{MR1314494}]\label{prop:MR}
For $\nu$ satisfying \eqref{eq:IC2} and $\mu$ an ergodic $\scrF$-invariant measure we have
\begin{equation}h_{\mu} (F \mid\pi)\le \sum_{\lambda^i_\mu>0}\lambda^i_\mu \ \dim E^i.\label{eq:MR} \end {equation}
\end{proposition}

We observe that the $\sigma$-algebra generated by $\pi^{-1}$ is $F$-invariant.  It follows (see for example \cite[Theorem I.4.2(9)]{MR1369243}) that 
$h_{\mu} (F \mid\pi) = h_{\mu} (F \inv \mid\pi).$
In particular, for the fiber-entropy of skew products  we have the reverse Margulis--Ruelle inequality 
$$h_{\mu} (F \mid\pi)\le \sum_{\lambda^i_\mu<0}-\lambda^i_\mu \ \dim E^i.$$
\subsubsection{Fiber-wise SRB measures}

It will follow from an extension of Pesin theory (see Section \ref{sec:PT}) that for any cocycle $\scrF$ satisfying the integrability hypothesis \eqref{eq:IC2} 
and any $\scrF$-invariant measure $\mu$ with positive fiber-entropy that for $\mu$-almost every $(\xi, x) \in \Sigma\times M$ there is an injectively immersed curve $W^u(\xi,x)\subset M_\xi$ tangent to $E^u(\xi,x)$, 
called the \emph{fiber-wise unstable manifold} at $(\xi,x)$.  
We say $\mu$ is \emph{fiber-wise SRB} if the conditional measures of $\mu$ along leaves of $W^u(\xi,x)\subset M_\xi$ are absolutely continuous.  (See 
Definition \ref{def:fiberwiseSRB} for precise statement).

As in the case of random dynamics (Proposition \ref{prop:SRBrandom}), we have the following characterization of fiber-wise SRB measures as those for which \eqref{eq:MR} is an equality.  
\begin{proposition}\label{prop:fiberSRB}
For $\nu$ satisfying 
\eqref{eq:IC2}, an ergodic  $\scrF$-invariant measure  $\mu$ is {fiber-wise SRB} if and only if 
$$h_\mu(F\mid\pi)= \sum_{\lambda^i_\mu>0}\lambda^i_\mu \ \dim E^i.$$
\end{proposition}

Proposition \ref{prop:fiberSRB} extends the result for i.i.d.\ random dynamical systems (Proposition \ref{prop:SRBrandom}) 
to the case of abstract skew products.  This formulation is proven in \cite{MR1646606}.

\subsection{Reinterpretation of Theorem \ref{thm:main}}\label{sec:skewreint}
Let the  $M$ and $\hat\nu$ be as in Section \ref{sec:resultsStationary} and let $\hat \mu$ be an ergodic, $\hat\nu$-stationary measure with $h_\mu(\MP^+(M, \hat\nu))>0.$   
  Let $F\colon \Sigma \times M\to \Sigma\times M$ denote the canonical skew product and let $\mu$ be the measure given by Proposition \ref{prop:mudef}.  It follows from Proposition \ref{prop:entropiessame} that $h_\mu(F\mid \pi)>0$.  By the Margulis--Ruelle inequality \eqref{eq:MR}, positivity of the fiber-entropy  implies that $\mu$ has two distinct exponents $-\infty< \lambda^s< 0<\lambda^u<\infty$ that, by construction, are equal to the exponents $\lambda_1, \lambda_2$ in \eqref{eq:pozzexp}.  We have a $\mu$-measurable splitting of $\Sigma\times TM$ into measurable line fields $$\Sigma\times T_{x}M = E^s_{(\xi,x)} \oplus E^u_{(\xi,x)} .$$

For $\sigma\in \{s,u\}$ and $(\xi,x) \in \Sigma\times M$  we write $E^\sigma_\xi(x)\subset TM$ for the subspace with $E^\sigma_{(\xi,x)} = \{\xi\} \times E^\sigma_\xi(x)$.  Projectivizing the tangent bundle $TM$, we obtain a measurable function  $$(\xi,x)\mapsto E^\sigma_\xi(x).$$

For $\xi = (\dots, \xi_{-2}, \xi_{-1}, \xi_0, \xi_1, \xi_2,\dots ) \in \Sigma$ write $\Sigmalocs(\xi)$ and $\Sigmalocu(\xi)$  for the \emph{local stable} and \emph{unstable sets}
$$\Sigmalocs(\xi): = \{ \eta\in \Sigma \mid \eta_i = \xi_i \text{ for all } i\ge 0\}$$
$$\Sigmalocu(\xi): = \{ \eta\in \Sigma \mid \eta_i = \xi_i \text{ for all } i< 0\}. $$
\newcommand\Gol{\mathcal G}
Write $\hat \Fol$ for the sub-$\sigma$-algebra of (the completion of) the Borel $\sigma$-algebra on $\Sigma$ containing sets that are \as saturated by local unstable sets:
$C\in \hat \Fol$ if and only if $C= \hat C \mod \nunaught^\Z$ where $\hat C$ is Borel in $\Sigma$ with $$\hat C  = \bigcup_{\xi\in \hat C }\Sigmalocu(\xi).$$
Similarly, we define $\hat \Gol$ to be the sub-$\sigma$-algebra of $\Sigma$ whose atoms are local stable sets.  
Writing $\B_M$ for the Borel $\sigma$-algebra on $M$ we define the $\sigma$-algebra on $X$ to be the $\mu$-completion of the algebras $\Fol = \hat \Fol \otimes \B_M$ and $\Gol:= \hat \Gol \otimes \B_M$.  

We note that, by construction, the assignments $\Omega\to \diff^2(M)$ given by $\xi \mapsto f_\xi$  and $\xi\mapsto \cocycle  [\xi][-1]$ are, respectively,  $\hat \Gol$- and $\hat \Fol$-measurable.  Furthermore, observing that the stable line fields $E^s_\xi(x)$ depend only on the value of $\cocycle$ for $n\ge 0$, we have the following  straightforward but crucial observation. 
\begin{proposition}\label{prop:measurability} 
The map  
$(\xi,x)\mapsto E^s_\xi(x)$ is $\Gol$-measurable and  the map 
$(\xi,x)\mapsto E^u_\xi(x)$ is $\Fol$-measurable.  
\end{proposition}

We have the following claim, which follows from the explicit construction of $\mu$ in \eqref{eq:mulim}.  
\begin{proposition} \label{prop:hopf}
The intersection $\Fol\cap \Gol$ is equivalent modulo $  \mu$ to the $\sigma$-algebra $\{\emptyset, \Sigma\}\otimes \B_M.$
\end{proposition}
\begin{proof}
Let $A\in \Fol \cap  \Gol$.  Since $A\in \Gol$, we have that $A \circeq \hat A$ where $\hat A$ is a Borel subset of $\Sigma\times M$ such that for any $(\xi,y)\in \hat A $ and $\eta \in \Sigmalocs(\xi)$, $$(\eta,y) \in \hat A.$$   

We write $\{\mu^\Fol_{(\xi,x)}\}$ and  $\{\mu^\Sigma_{(\xi,x)}\}$, respectively, for  families of conditional probabilities given by  the partition of $\Sigma\times M$ into atoms of $\Fol$  and the partition $\{ \Sigma \times \{x\} \mid x\in M\}$ of $\Sigma\times M$. 
It follows from the construction of $\mu$ given by  \eqref{eq:mulim} that $\mu^\Fol_{(\xi,x)}$ may be taken to be the form 
\begin{align}\label{eq:thisone} d\mu^\Fol_{(\xi,x)}(\eta,y) = d\nunaught ^\N(\eta_0, \eta_1,  \dots) \delta_x(y) \delta _{(\xi_{-1})} (\eta_{-1})
 \delta _{(\xi_{-2})} (\eta_{-2})\dots \end{align}
for every $(\xi,x) \in X$.   

Since $A\in \Fol$ we have $\hat A \in \Fol$.  Thus, for $\mu$-\ae $(\xi,x)\in \hat A$, $$\mu^\Fol_{(\xi,x)}(\hat A)= 1.$$
Furthermore, it follows from \eqref{eq:thisone} and the form of $\hat A$ that if $$\mu^\Fol_{(\xi,x)}(\hat A)= 1$$ then $$\mu^\Fol_{(\xi',x)}(\hat A)= 1$$ for any $\xi' \in \Sigma$.  It follows that $$\mu^\Sigma_{(\xi,x)}\hat A = 1$$ for \ae $(\xi,x) \in \hat A$. 
In particular, $\hat A\circeq   \Sigma \times \td A$ for some set $\td A\in \B_M$. 
\end{proof}

We remark that if $\xi$ projects to $\omega$ under the natural projection $\Sigma\to \Sigma_+$, then the subspace $E^s_\xi(x)$ and the subspace  $E^s_\omega(x)$ given by Proposition \ref{prop:OMT} coincide almost surely.  
It then follows from Proposition \ref{prop:hopf} that  the bundle $E_\omega^s(x)$ in  Theorem \ref{thm:main} is non-random if and only if the bundle $E^s_\xi(x)$ is $\Fol$-measurable.  
Thus, Theorem \ref{thm:main} follows from the following.
\begin{theorem} \label{thm:main2}\label{thm:skewproduct}
Let $\hat\nu$ and $\hat\mu$ be as in Theorem \ref{thm:main}.  
Let $F\colon \Sigma\times M \to \Sigma \times M$ be the canonical skew product let $\mu$ be as in Proposition \ref{prop:mudef}.  Assume  the fiber entropy $h_\mu(F\mid \pi)$ is positive.  Then either $(\xi, x) \mapsto E^s_\xi(x)$ is $\F$-measurable, or $\mu$ is fiber-wise SRB.
\end{theorem}


\subsection{Statement of results: abstract skew products}
To prove Theorem \ref{thm:main} we will introduce a generalization of Theorem \ref{thm:main2}, the proof of which consumes Sections \ref{sec:MainLemma}--\ref{sec:lemmaproof}. 
Let $\theta\colon (\Omega, \B_\Omega, \nu) \to (\Omega, \B_\Omega, \nu) $ be as in Section \ref{sec:absSkew}.
Let $M$ be a closed $C^\infty$ surface and let $\scrF$ be a cocycle generated by a $\nu$-measurable map $\xi\mapsto f_\xi$ satisfying the integrability  hypothesis 
\eqref{eq:IC2}.  
Fix $\mu$ an ergodic, $\scrF$-invariant, Borel probability measure on $X= \Omega\times M$.  
We assume $h_\mu(F\mid\pi)>0$ so that $DF$ has two exponents $\lambda^s$ and $\lambda^u$, one  of each sign.  

We say a sub-$\sigma$-algebra $\hat \F \subset \B_\Omega$ is \emph{increasing} if $$\theta(\hat\Fol )= \{ \theta(A)\mid A\in \hat \Fol\}\subset \hat \Fol.$$ That is,   $\hat \Fol$ is increasing if the partition into atoms is an increasing partition in the sense of \cite{MR819556}.
(Alternatively,  $\hat \Fol$ is  increasing if  the map $\theta\inv\colon \Omega\to \Omega$ is $\hat\Fol$-measurable.)
As a primary example, the sub-$\sigma$-algebra of $\Sigma$ generated by local unstable sets is increasing (for $\sigma\colon \Sigma \to \Sigma$).  

Let $\hat \F$ be an increasing sub-$\sigma$-algebra and write $\F$ for the $\mu$-completion of $\hat \F \otimes \B_M$ where $\B_M$ is the Borel algebra on $M$.  We note that $\F$ is an increasing sub-$\sigma$-algebra of $\B_X$.  
Let $\{\mu_\xi\}_{\xi\in \Omega}$ denote the family of conditional probability measures with respect to the partition induced by the projection $\pi\colon X\to \Omega$.  Using the canonical identification of fibers $M_\xi= \{\xi\}\times M$ in $X$ with $M$, by an abuse of notation we consider the map $\xi\mapsto \mu_\xi$ as a measurable map from $\Omega$ to the space of Borel probabilities on $M$.  
As in the previous section, to compare stable distributions in different fibers over $\Omega$ write $E_\xi^s(x)\subset T_xM$ for the subspace with $E^s(\xi,x) = \{\xi\}\times E_\xi^s(x)$.  We then consider $(\xi,x) \mapsto  \{\xi\}\times E_\xi^s(x)$ as a measurable map from $X$ to the projectivization of $TM$.

With the above setup, we now state the main theorem of the paper.   

\label{sec:MainTheoremABSskew}
\begin{theorem} \label{thm:main3}\label{thm:skewproductABS}
Assume  $\mu$ has positive fiber entropy  and that 
\begin{enumerate}
\item $\xi\mapsto \cocycle    [\xi] [-1]$ is $\hat \F$-measurable, and
\item $\xi\mapsto \mu_\xi$ is $\hat \F$-measurable.
\end{enumerate}
Then either $(\xi,x)\mapsto E^s_{\xi}(x)$ is $\F$-measurable or $\mu$ is fiber-wise SRB.  
\end{theorem}

We recall that in the case that $F$ is the canonical skew product for a random dynamical system and $\hat \Fol$ is the sub-$\sigma$-algebra generated by local unstable sets, writing $\xi = (\dots, f_{-1}, f_0, f_1, \dots)$ the $\hat\F$-measurability of $\xi\mapsto \cocycle    [\xi] [-1] = (f_{-1})^{-1}$ follows from construction.  The $\hat\F$-measurability of $\xi\mapsto \mu_\xi$ follows from the construction of the measure $\mu$ given by \eqref{eq:mulim} in Proposition \ref{prop:mudef}.  
Theorems \ref{thm:skewproduct} and \ref{thm:1}  then follow immediately from Theorem \ref{thm:skewproductABS}.    


\section{Background and notation} \label{sec:BG}
In this section, we continue work in the setting introduced in Sections \ref{sec:absSkew} and \ref{sec:MainTheoremABSskew}.  
We outline extensions of a number of standard facts from the theory of nonuniformly hyperbolic diffeomorphisms to the setting of the fiber-wise dynamics for skew products.  
As previously observed, positivity of the fiber-wise metric entropy $h_\mu(F\mid \pi)$ implies that we have at least one Lyapunov exponent of each sign
$\lambda^s<0<\lambda^u$.  
\newcommand{\eps}{\oldepsilon_0}
For the remainder, fix $0< \eps < \min\{1,\lambda^u/200, -\lambda^s/200\} .$ 

\subsection{Fiber-wise Pesin Theory}\label{sec:PT}
We present an extension of the classical theory of stable and unstable manifolds for nonuniformly hyperbolic diffeomorphism \cite{MR0458490} to the fiber-wise
 dynamics of skew products.  

\subsubsection{Subexponential estimates} 
We have the following standard results that follow from the integrability hypothesis 
\eqref{eq:IC2} and 
tempering kernel arguments (c.f.\ \cite[Lemma 3.5.7]{MR2348606}.)
\begin{prop}\label{prop:tempered}
There is subset $\Omega_0\subset \Omega$ with $\nu(\Omega_0)= 1$ and measurable function $D\colon \Omega_0\to (0,\infty)$ such that for $\nu$-a.e. $\xi\in \Omega_0$ and $n\in \Z$.
\begin{enumerate}
\item $|f_{\theta^n(\xi)}|_{C^1} \le  e^{|n|\eps}D(\xi)$
\item $\lip(Df_{\theta^n(\xi)}) \le   e^{|n|\eps}D(\xi)$.  
\end{enumerate}
\end{prop}
\noindent Here $\lip(Df_\xi)$ denotes the Lipschitz constant of the map $x\mapsto D_xf_\xi$ for fixed $\xi.$
 
We also have the  following standard result in the theory of Lyapunov exponents.  
\begin{prop}\label{prop:growth}
There is a measurable  function $L\colon X\to (0,\infty)$   
such that for $\mu$-\ae $(\xi,x)\in X$ and $n \in \Z$
\begin{enumerate}
\item For $v\in E^s(\xi, x)$, 
	\[	 L(\xi,x)\inv \exp(n\lambda^s- |n|\tfrac 1 2 \eps)\|v\|\le  \|Df_{\xi}^n v\|
		\le L(\xi,x) \exp(n\lambda^s+ |n|\tfrac 1 2 \eps)\|v\| . \]
\item For $v\in E^u(\xi, x)$, 
\[	L(\xi,x)\inv  \exp(n\lambda^u- |n|\tfrac 1 2 \eps)\|v\|\le  \|Df_{\xi}^n v\|
		\le L(\xi,x) \exp(n\lambda^u+ |n|\tfrac 1 2 \eps)\|v\|.   \]
\item $\angle\left(E^s(F^n(\xi, x)), E^u(F^n(\xi,x))\right)  >\dfrac{1}{L(\xi, x)} \exp (-|n|\epsilon) $.
\end{enumerate}
\end{prop}
\noindent Here $\angle$ denotes the Riemannian angle between two subspaces.

\subsubsection{Stable manifold theorem}
We now consider fiber-wise stable and unstable manifold.  
The existence stable manifolds for diffeomorphisms with non-zero exponents is due to Pesin in the deterministic case \cite{MR0458490}. 
In the case of random dynamical systems, the statements and proofs hold with minor modifications.  
We adapt the version of the stable manifold theorem from \cite{MR1369243}.  

For $n\in \N$ write $$E_n(\xi,x) = E^s_{\theta^n(\xi)}(f_\xi^n(x))\quad \quad H_n(\xi,x) = (E^s_{\theta^n(\xi)}(f_\xi^n(x)))^\perp$$
where $E^\perp$ is the orthogonal complement in the background Riemannian metric.  
\begin{theorem}[Local stable manifold theorem]\label{thm:locPesinStab}
\item For $\mu$-\ae $(\xi,x)\in X$ 
there are numbers $\beta, \gamma, \alpha$ (depending measurably on $(\xi,x)$) such that for any $n\ge 0$ there are  
 $C^{1,1}$ functions $$h_{(\xi,x),n}^s\colon B^s(0, \alpha \exp(-5n\eps))\subset E_n(\xi,x) \mapsto H_n(\xi,x)$$ with  \begin{enumerate}
	\item $h_{(\xi,x),n}^s(0)= 0$
\item 				$Dh_{(\xi,x),n}^s(0)= 0$
	\item $\lip (h_{(\xi,x),n}^s) \le \beta \exp(7n\eps)$ and $\lip (Dh_{(\xi,x),n}^s) \le \beta \exp(7n\eps).$ 
\end{enumerate}
Setting 
\item  
$V^s_n(\xi, x) := \exp_{f_\xi^n (x)} \left( \mathrm{Graph}\left({ h^s_{(\xi, x),n}}  
\right)\right)$ we have 
\begin{enumerate}[resume]
\item $f_{\theta^n(\xi)}(V^s_n(\xi, x)) \subset V^s_{n+1}(\xi,x)$
\item for $z,y\in V^s_0(\xi,x)$
$$d^s_n(F^n(\xi, z), F^n(\xi(y))\le
	\gamma(\xi, x) \exp((\lambda^s+ 7\eps)n)d_0^s(y,z)$$
	where $d^s_n $ denotes the induced Riemannian distance in $V^s_n(\xi, x)$.
\end{enumerate}
\end{theorem}
\def\loc{\mathrm{loc}}
\def\locPStab{V^s_{\mathrm {loc}}(\xi,x)}
\def\locPStaby{V^s_{\mathrm {loc}}(\xi,y)}
We define $\locPStab = V^s_0(\xi,x)\subset M$ to be the \emph{local stable manifold} at $x$ for $\xi$.  
We similarly define local unstable manifolds.  

We define the \emph{global stable} and \emph{unstable manifolds} at $x$ for  $\xi$ by
\begin{align}
W^s_\xi(x): = \{ y\in M \mid \limsup_{n\to \infty} \dfrac 1 n \log d (f_\xi^n(x), f_\xi^n(y) )<0\}\\
W^u_\xi(x): = \{ y\in M \mid \limsup_{n\to -\infty} \dfrac 1 n \log d (f_\xi^n(x), f_\xi^n(y) )<0\}.
\end{align}
We have  for $\mu$-\ae $(\xi,x)$ that $W^s_\xi(x)$ is a $C^{1,1}$-injectively immersed curve tangent to $E^s_\xi(x)$.  

For $p =(\xi,x)\in X$ we write 
$$W^s(p) = W^s(\xi,x):= \{\xi\}\times W^s_{\xi}(x), \quad W^u(p) = W^u(\xi,x):= \{\xi\}\times W^u_{\xi}(x)$$ for the associated \emph{fiber-wise stable} and \emph{unstable} manifolds in $X$.

\subsubsection{Lyapunov norm}
\newcommand{\lyap}[1]{\left\vvvert #1\right\vvvert}
Although the derivative cocycle $DF$ is hyperbolic on long time scales, it is convenient at times to use a norm on $TX$ adapted to the dynamics so that hyperbolicity is seen after a single iterate.  The drawback of such a norm will be that it is defined only almost everywhere, varies with $\xi$, and depends measurably on $ x\in M_\xi$,  whereas the original Riemannian metric is constant in $\xi$ and smooth in $x$.  

For  $(\xi,x)\in X$ and $v\in E^s_\xi(x), w\in E^u_\xi(x)$ define 
\begin{align}
&\left( \lyap{v}_{\eps, (\xi,x)}^s\right)^2:= \sum_{n\in \Z}   \|Df^n_\xi v\|^2 e^{ -2 \lambda^s n -  2 \eps |n|}\\
&\left(\lyap w _{\eps, (\xi,x)}^u\right)^2:= \sum_{n\in \Z}   \|Df^n_\xi w\|^2 e^{-  2\lambda^u n - 2 \eps |n|}
\end{align}

It follows from Proposition \ref{prop:growth} that the sums above converge almost everywhere.   Declaring for $v^u\in E^u_\xi(x)$ and $ v^s\in E^s_\xi(x)$ that 
	$$\lyap{v^u+ v^s}^2_{\eps, (\xi,x)}= \lyap{v^s}^2_{\eps, (\xi,x)}+\lyap{v^u}^2_{\eps, (\xi,x)}$$
we obtain a measurable family of norms on $TM$ called the \emph{Lyapunov Norm}.

Measured in the norm $\lyap{\cdot}_{\eps, (\xi,x)}$, the fiber-wise dynamics becomes uniformly hyperbolic via  the following estimate.
\begin{prop}
For $(\xi,x)\in X$ satisfying Proposition \ref{prop:growth}, $v\in E^s_\xi(x), w\in E^u_\xi(x)$, and $k\in \Z$ we have 
\begin{align*}
e^{k\lambda^s -  |k|\eps} \lyap v_{\eps, (\xi,x)}^s\le &
	 \lyap{Df^k_\xi v}_{\eps, F^k(\xi,x)}^s
	 \le e^{k\lambda^s +  |k|\eps} \lyap v_{\eps, (\xi,x)}^s\\
e^{k\lambda^u  - |k|\eps} \lyap w_{\eps, (\xi,x)}^u\le &
	 \lyap{Df^k_\xi w}_{\eps, F^k(\xi,x)}^u
	 \le e^{k\lambda^u  + |k|\eps} \lyap w_{\eps, (\xi,x)}^u.
\end{align*}
\end{prop}

\begin{proof}
We show the first set of inequalities.  For $v\in  E^s_\xi(x)$ and $k\in \Z$, writing $\ell = n+k$ we have 
\begin{align*}\left(\lyap{Df^k_\xi v}^s_{\eps, F^k(\xi,x)}\right)^2:=& \sum_{n\in \Z}   \|Df^{n+k}_\xi v\| ^2e^{ - 2 n  \lambda^s  - 2 |n|\eps }\\
=& \sum_{\ell\in \Z}  \|Df^{\ell}_\xi v\| ^2e^{ -2(\ell-k)\lambda^s  -2 |\ell-k| \eps}.\\
\end{align*}
Noting that $$ e^{-2 |\ell| \eps- 2|k| \eps}\le e^{-2 |\ell-k| \eps}\le e^{-2 |\ell| \eps+ 2|k| \eps}$$
we have 
\begin{align*}\sum_{\ell\in \Z}   \|Df^{\ell}_\xi v\|^2 e^{ -2\ell  \lambda^s  -  2|\ell|\eps }
&\left(e^{ 2k\lambda^s -2|k| \eps} \right)\le  \lyap{Df^k_\xi v}_{\eps, F^k(\xi,x)}^2\\&\le \sum_{\ell\in \Z}   \|Df^{\ell}_\xi v\| ^2e^{ -2\ell  \lambda^s  - 2 |\ell|\eps }\left(e^{ 2k\lambda^s +2|k| \eps} \right) \ .\end{align*}
Thus 
$$ e^{ 2k\lambda^s -2|k| \eps}  \left(\lyap{v}^s_{\eps, (\xi,x)}\right)^2\le\left( \lyap{Df_\xi v}^s_{\eps, F^k(\xi,x)} \right)^2\le e^{ 2k\lambda^s +2|k| \eps}  \left( \lyap{v}^s_{\eps, (\xi,x)}\right)^2. \qedhere $$

\end{proof}

When it is clear from context, we will drop the majority of  sub- and superscripts from the Lyapunov norm.  

\subsection{Affine parameters}
Since each stable and unstable manifold in $X$ is a curve, it has a natural parametrization via the Riemannian arc length.  
We define an alternative parametrization, defined on almost every stable manifold, that conjugates the non-linear dynamics $\restrict {f^n_\xi}{\stabM x \xi}$ and the linear dynamics $\restrict {Df^n_\xi}{E^s _\xi(x)}$.  We sketch  the construction and refer the reader to \cite[Section 3.1] {MR2261075} for additional details.  
 
\begin{prop}\label{prop:Stabman}
For almost every $(\xi,x)$ and any $y\in \stabM x \xi$, there is a $C^{1,1}$ diffeomorphism $$H^s_{(\xi,y)} \colon \stabM x \xi \to T_y \stabM x \xi$$ 
such that 
\begin{enumerate}
\item  restricted to   $\stabM x \xi$ the parametrization  
 intertwines the  nonlinear dynamics $f_\xi$ with the differential $D_yf_\xi$: 
$$ D_y f_\xi \circ  H^s_{(\xi,y)} = H^s_{F(\xi,y)} \circ f_\xi;$$

\item $H^s_{(\xi,y)} (y)= 0$ and $D_yH^s_{(\xi,y)} = \id$;
\item if $z\in \stabM x \xi$ then the change of coordinates   
\[H_{(\xi,y)} ^s \circ \left(H_{(\xi,z)} ^s\right)\inv \colon T_z \stabM x \xi \to T_y \stabM x \xi\] 
is an affine map  with derivative 
$$D_v\left(H_{(\xi,y)} ^s \circ \left(H_{(\xi,z)} ^s\right)\inv\right) = \rho_{(\xi,y)}(z)$$ for any $v\in T_z \stabM x \xi$  
where $ \rho_{(\xi,y)}(z)$  is defined below.  
\end{enumerate}
\end{prop}
We take $(\xi,x)$ to be in the full measure set such that for any $y,z\in \stabM x\xi$ there is some $k\ge 0$ with $f_\xi ^k(z)$ and $f_\xi ^k(y)$ contained in $V^s_{\mathrm {loc}}(F^k(\xi,x))$ and sketch the construction of $H_{(\xi,y)}^s$.   
First consider any  $y,z\in V^s_{\mathrm {loc}}(\xi,x)$  
and define
$$J(\xi, z):= \|D_zf_\xi  v\| \cdot \|v\|\inv$$ for any non-zero $v\in T_z\unstM x \xi$ 
where $\|\cdot\|$ denotes the Riemannian norm on $M$.
We define 
\begin{equation}\label{eq:rho}\rho_{(\xi,y)}(z) := \prod_{k = 0}^\infty\dfrac{J(F^k(\xi, z))}{J(F^k(\xi,y))}\end{equation}
Following  \cite[Section 3.1] {MR2261075}, we have that the right hand side of \eqref{eq:rho} converges uniformly in $z$ to a Lipschitz function.  The only minor modification needed in our setting comes from the subexponential growth of $Df_\xi$ and its Lipschitz constant along orbits given by Proposition \ref{prop:tempered}. 
We may extend the definition of $\rho_{(\xi,y)}(z)$ to any $z,y\in \stabM x \xi$ using that $f_\xi ^k(z)$ and $f_\xi ^k(y)$ are contained in $V^s_{\mathrm {loc}}  (F^k(\xi,x))$ for some $k\ge 0$. 



We now define the affine parameter $H^s_{(\xi,y)}\colon \stabM x \xi\to T_y\stabM x \xi$ as follows.  
We define $H^s_{(\xi,y)}$ to be orientation preserving and  
$$|H^s_{(\xi,y)}(z)|:= \int_y^z \rho_{(\xi,y)}(t) \ dt $$
where $\int_y^z \psi(t) \ dt$ is the integral of the function $\psi$, along the curve from $y$ to $z$ in $\stabM x \xi$,  with respect to the Riemannian arc-length on $\stabM x \xi$.  

It follows from computations in \cite[Lemma 3.2, Lemma 3.3] {MR2261075} that the map $H^s_{(\xi,y)}$ constructed above satisfies the properties above. 

We similarly construct unstable affine parameters $H^u_{(\xi,x)}$ with analogous properties.  

\subsubsection{Parametrization of local stable manifolds}
We use the affine parameters $H^s$ and the background Riemannian norm on $M$ to parametrize local stable manifolds.  
For $p=(\xi,x)\in X$ such that  affine parameters are defined, write 
$$\displaystyle \locstabM  x  \xi := (H^s_x)\inv\left (\{v\in E^s_\xi(x) \mid \|v\| < r\}\right) $$ for the local stable manifold in $M$ and 
$$ \displaystyle \locstabp p =  \locstab  x  \xi := \{\xi\}\times  \locstabM  x  \xi$$ for the corresponding  fiber-wise local stable manifold.  
We use similar notation for local unstable manifolds.

\subsubsection{A bound on distortion}
 
By Lusin's theorem\footnote{Recall that $\Omega$, and hence $X$, are Polish.} for any $\delta>0$ there is a compact subset $\Lambda'\subset \Omega_0\times M$ of measure $\mu(\Lambda')>1-\delta$ 
such that the maps $(\xi,x)\mapsto E^s(\xi,x)$,  $(\xi,x)\mapsto  \locPStab$ and  $(\xi,x)\mapsto  V^u_{\mathrm{loc}}(\xi,x)$ are continuous on $\Lambda'$.  
We will use the following estimates, which follow from the construction of stable manifolds and standard arguments.  

\def\rhere{\hat r}
\begin{lemma}\label{lem:contsets} \label{lem:standardcrap} 

There is a set $\Lambda\subset \Lambda'$ with $\mu(\Lambda)>1-2\delta$, and $\gamma>0$ and $\hat r>0$ 
such that for 
and for any $(\xi,x) \in \Lambda$ and $(\xi,y)\in \Lambda$ with $d(x,y)<\gamma$
the intersection $\locstabM[ \rhere] x \xi \cap \locunstM[ \rhere] y \xi$ is a singleton and the intersection is uniformly transverse.    

Furthermore there is a $C_1>0$   such that
for $(\xi,x) \in \Lambda$ and $(\xi,y)\in \Lambda$ with $d(x,y)<\gamma$, setting $z = \locunstM[ \rhere] x \xi \cap \locstabM[ \rhere] y \xi$
\begin{enumerate}
\item \label{item8:1} $\displaystyle\dfrac 1 {C_1}\le  \|\restrict {D_x\cocycle[\xi][-n]  }{T_x\locunstM[ \rhere] x \xi }\| \cdot \|\restrict{D_z\cocycle[\xi][-n] }{T_z\locunstM[ \rhere] x \xi } \|\inv \le C_1 $ for all $n\ge 0$
\item \label{item8:2} $\dfrac1 {C_1}\le  \displaystyle \|\restrict {D _y \cocycle }{T_y\locunstM[ \rhere] y \xi} \| \cdot \|\restrict{D_z \cocycle }{T_z\locunstM[ \rhere] x \xi} \|\inv \le C_1$ for all $n\ge 0$.  
\end{enumerate}
\end{lemma}

The estimates follow from the fact that the pairs $\cocycle[\xi][n] (y)$ and $\cocycle [\xi][n] (z)$,   $\cocycle [\xi][-n](x)$ and $\cocycle[\xi][-n](z)$, and $D _y \cocycle (T_y\locunstM[ \rhere] y \xi)$ and $D_z \cocycle (T_z\locunstM[ \rhere] x \xi)$ are exponentially asymptotic while $|f_\xi|_{C^1}$, $\lip(Df^n_{\xi})$, and the Lipschitz constant for the variation of the tangent spaces to $\cocycle [\xi] [-n] (\locunstM x \xi )$ grow subexponentially for  $\xi\in \Omega$ and $(\xi,x)$ satisfying Proposition \ref{prop:Stabman}.  The existence of such a $\Lambda'$ then follows from Lusin's theorem.

\subsection{Families of conditional measures}
\label{sec:condmeas}
The family of fiber-wise unstable manifolds $\{W^u(p)\}_{p\in X}$ forms a partition of a full measure subset of $X$.  However, such a partition is generally non-measurable. To define conditional measures we consider  a measurable partition $\P$ of $X$  such for $\mu$-\ae $p\in X$ there is an $r$ such that $\locunstp p\subset  \P(p)\subset W^u(p)$. 
Such a partition is said to be \emph{$u$-subordinate}.   
 Let $\{\td \mu^{\P}_p\}_{p\in X}$ denote a family of conditional probability measures with respect to such a  partition $\P$. 

\begin{definition}\label{def:fiberwiseSRB}
An $\scrF$-invariant measure $\mu$ is \emph{fiber-wise SRB} if for any $u$-subordinate measurable partition $\P$ with corresponding family of conditional measures $\{\td \mu^{\P}_p\}_{p\in X}$, the measure $\td \mu^\P_p$ is absolutely continuous with respect to Riemannian volume on $\unstp p$ for \ae $p$. 
\end{definition}

In the setting introduced in Sections \ref{sec:resultsStationary} we have the following.  
\begin{definition}\label{def:SRB}
Let $M$ be a closed manifold, $ \nunaught$ a Borel measure on $\diff^2(M)$  
and  let $\hat\mu$ be a $\hat \nu$-stationary probability measure.  We say $\hat\mu$ is \emph{SRB} if the measure $\mu$ given by Proposition \ref{prop:mudef} is fiber-wise SRB for the associated canonical skew product \eqref{eq:skewdefn}.
\end{definition}

\begin{remark}
In fact, it follows from the proof of Proposition \ref{prop:fiberSRB} that $\mu$ is SRB if and only if the conditional measures $\{\td \mu^{\P}_p\}_{p\in X}$ are equivalent to Riemannian volume on $\unstp p$.  See, for example, \cite[Corollary 6.1.4]{MR819556}.
  \end{remark}
 
 Let $\P_1$ and $\P_2$ be two measurable partitions of $X$ subordinate to $\{W^u(p)\}_{p\in X}$.  It follows that, for $\mu$-\ae $p$, the conditional measures $\td \mu^{\P_1}_p$ and  $\td \mu^{\P_1}_p$ coincide, up to a normalization factor, on the intersection $\P_1(p) \cap \P_2(p)$.  

\newcommand{\scond}[2][(\xi, x)]{{ \mu^s_{#1}}\left({#2}\right)}
\newcommand{\ucond}[2][(\xi, x)]{ \mu^u_{#1}\left({#2}\right)}

By fixing a normalization, we  define a \emph{locally-finite}, infinite measure $ \mu^u_p$ on each curve $\unstp p$.  Such a measure will be locally-finite in the internal topology of $\unstp p$ induced, for instance, by the affine parameters.  To construct such a family of measures, first consider a countable sequence of measurable, $u$-subordinated partitions $\P_n$ with the property that for any compact (in the internal topology of $\unstp p $) subset $K\subset \unstp p $ there is a $\P_n$ with  $$\locunstp [1] p \subset \P_n(p)\quad \text{and} \quad K \subset \P_n(p).$$
Then, for almost every $p\in X$ (such that $p$ is contained in the support of $\td \mu^{\P_n}_p$ for every $n$) we may define $$\restrict{ \mu^u_p}{K}:= \dfrac 1 {\td \mu^{\P_n}_p( \locunstp[1] p) } {\td \mu^{\P_n}_p.}$$
One verifies  that for various  choices of $K$ and $\P_n$ the above definition is coherent and uniquely defines $\mu^u_p$.  
We similarly define locally-finite families ${\mu^s_p}$ of infinite measures on  the fiber-wise stable manifolds.  

We remark that the fiber entropy $h_\mu(F\mid \pi)$ is positive if and only if the measures $\mu^u_p$ and $\mu^s_p$ are non-atomic for almost every $p$.

\subsection{Orientation, trivialization, and $\F$-measurable geometric structures}
We recall the main hypotheses in Theorem \ref{thm:skewproductABS}: $\hat \Fol$ is an increasing sub-$\sigma$-algebra of $\B_\Omega$, $\xi\mapsto \cocycle [\xi][-1]$ and $\xi\mapsto \mu_\xi$ are  $\hat \Fol$-measurable.  Since $\hat \Fol\subset \theta\inv (\hat\Fol) $, it follows  for any $n\ge 0$ that $$\xi \mapsto \cocycle[\theta^{-n}(\xi)] [-1]$$ is $\hat \Fol$-measurable whence $$\xi\mapsto \cocycle [\xi][-n]$$ is $\hat \Fol$-measurable for any $n\ge 1$.  
Furthermore, we have that  unstable subspaces and manifolds
		$$E^u_\xi(x)= \{v\in T_xM \mid \lim_{n\to - \infty}\dfrac 1 n |D_xf_\xi^n v| = \lambda^u\}$$ 
		$$\unstM x \xi = \{ y\in M \mid \limsup_{n\to -\infty} \dfrac 1 n \log d (f_\xi^n(x), f_\xi^n(y) )<0\}$$
 depend only on the past dynamics $\cocycle [\xi] [-n], n\ge 1$ and hence are $\F$-measurable.  
 Furthermore, since the family ${\mu_\xi}$ is assumed to be $\F$-measurable and since the locally-finite families $\{\mu^u_{(\xi,x)}\}$ are normalized using the affine parameters $H^u$, which are defined using only the past dynamics $\cocycle [\xi] [-n], n\ge 1$, it follows that 	$(\xi,x)\mapsto  \mu^u_{(\xi,x)}$ is $\F$-measurable.

\subsubsection{Orientation and trivialization}
Consider the measurable subbundle $\E^u\to X$ of the vector bundle $TX\to X$ whose fiber at $(\xi,x)$ is $E^u(\xi,x)$.  
By the $\F$-measurability of $(\xi,x)\mapsto E^u_\xi(x)$ we may  choose an $\F$-measurable assignment $(\xi,x)\mapsto v(\xi,x) \in E^u_\xi(x)\sm \{0\}$ with $\|v(\xi,x)\| = 1$. 
It follows that  $(\xi,x) \mapsto \left((\xi,x), v(\xi,x)\right)$ gives an $\F$-measurable orientation on $\E^u\subset T_xM$.  
We define  $\I \colon \E^u\to \R$
\begin{equation}
\I\colon\!\!\! \left((\xi,x), tv(\xi,x)\right) \mapsto   t.  \label{eq:trivialization}
\end{equation}

For $p\in X$, we write $\I_{p}\colon E^u(p)\to \R$ for the restriction of $\I$ to $E^u(p)$.  
We thus obtain a $\F$-measurable trivialization $\E^u\to X\times \R$
$$ \left(p, v\right)\mapsto \{p\} \times \I_p(v).$$

We also define a map from $X$ to the space of $C^1$ embeddings of $\R$ into $M$ by \begin{align}p \mapsto \left(t \mapsto (H^u_p)\inv \circ I_p\inv (t)\right).\label{eq:defofparam}\end{align}
Since the affine parameters $H^u_p$ defined on unstable manifolds depend only on the past dynamics $\cocycle [\xi] [-n], n\ge 1$, it follows that the map \eqref{eq:defofparam} is $\F$-measurable.  

We summarize the above.  
\begin{prop}\label{prop:transferable}
The geometric structures 
	$(\xi,x ) \mapsto E^u_\xi(x)$,
	$(\xi,x)\mapsto \unstM x \xi$, 
	$(\xi,x)\mapsto  \mu^u_{(\xi,x)}$, and \eqref{eq:defofparam}
are $\F$-measurable.

\end{prop}
\def\nuhat{\hat \nu} 
\subsection{The family $\bar \mu_{(\xi,x)}$}
Using the affine parameters  $H^u_{(\xi,x)}\colon \unstM x \xi \to E^u _\xi (x)$ and the trivialization $\I\colon \E^u \to  \R$ we define a family of locally-finite  Borel  measures on $\R$ by 
\begin{equation}
\bar \mu_{(\xi,x)}:= \left(\I \circ H^u_{(\xi,x)}\right)_*\mu_{(\xi,x)} ^u. \label{eq:mubardefn}
\end{equation}
We equip the space of locally-finite  Borel  measures on $\R$  with its standard Borel structure (the dual topology to compactly supported continuous functions).
We thus obtain a measurable function from $X$ to the locally-finite  Borel  measures on $\R$. 
Since the family of measures $p\mapsto \mu_{p} ^u$ is $\Fol$-measurable, it follows that $$p\mapsto \bar \mu_p$$ is $\Fol$-measurable.  

The family $\{\bar \mu_p\}_{p\in X}$ will be our primary focus in the sequel.  In particular, the SRB property of $\mu$ will follow by showing that for $\mu$-\ae $p$, the measure  $\bar \mu_p$ is the Lebesgue measure on $\R$ (normalized on $[-1,1]$).

\section{The main lemma and proof of Theorem \ref{thm:skewproduct}.} \label{sec:MainLemma}
The primary technical tool used to the prove  Theorem \ref{thm:main2} is the following lemma.
Given two locally finite measures $\eta_1$ and $\eta_2$ on $\R$ we write $\eta_1 \simeq \eta_2$ if there is some $c>0$ with $\eta_1 = c\eta_2$.  

\begin{lemma}[Main Lemma]\label{lem:main}
Assume in Theorem \ref{thm:skewproduct} that $(\xi,x)\mapsto E^s_\xi(x)$ is {not} $\F$-measurable. 
Then there exist  constants $M>0$ and $1>\delta_0>0$ such that for every { sufficiently small} $\epsilon>0$ there exists a 
compact set $G_\epsilon\subset X$  
with 
$$\mu(G_\epsilon)\ge \delta_0$$ satisfying the following:
  
For any $q\in G_\epsilon$ there is an affine map 
	$$\psi\colon \R\to \R$$ 
	 with 
	\begin{enumerate}
	\item $\dfrac{1}{M}\le | D\psi|\le {M}$;
	\item $\dfrac{\epsilon}{M} \le |\psi(0)|\le M\epsilon$;
	\item $\psi_*\bar \mu_q\simeq\bar \mu_q$. 
	\end{enumerate}
\end{lemma}

The proof of Lemma \ref{lem:main} occupies Sections \ref{sec:lemmas} and \ref{sec:lemmaproof}.  We finish this section with the proof of Theorem \ref{thm:skewproductABS} assuming Lemma \ref{lem:main}.
Set $$G:= \{ q\in X \mid q\in G_{1/N} \ \text{for\ infinitely\ many $N\in\N$}\}.$$
Under the hypotheses of Lemma \ref{lem:main} we have
 $\mu(G)\ge \delta_0$. 

\subsection{Proof of Theorem \ref{thm:skewproductABS}}
Theorem \ref{thm:skewproductABS} follows from Lemma \ref{lem:main} by  standard arguments.  
We sketch these  below and referring to \cite{MR2261075} for more details.  
\def\Aff{\mathcal{A}}
 
 \begin{lemma}\label{lem:translation1}
Under the hypotheses of Lemma \ref{lem:main}, for \ae $p\in X$, $\bar \mu_p$ is invariant under the group of translations.  In particular, for \ae $p\in X$, $\bar \mu_p$ is the Lebesgue measure on $\R$ normalized on $[-1,1]$.  
\end{lemma}

\begin{proof} 
\def\AffR{\mathrm{Aff}(\R)}
Let $\AffR$ denote the group of invertible affine transformations of $\R$.  
For $p\in X$, let $\Aff(p)\subset \AffR$ be the group of affine transformations $\psi\colon \R\to \R$ with $$\psi_*\bar \mu_p \simeq \bar\mu_p.$$  
We have  that $\Aff(p)$ is a closed subgroup of $\AffR$.  (See the proof of \cite[Lemma 3.10]{MR2261075}.)
By Lemma \ref{lem:main}, for $p\in G$, $\Aff(p)$ contains elements of the form $t\mapsto \lambda_j t+ v_j$ with $|v_j|\to 0$ as $j\to \infty$ and $\lambda_j\in \R$ such that  $|\lambda_j|$ is  uniformly bounded away from $0$ and $\infty$.  
Then, for $p\in G$,  $\Aff(p)$ contains at least one map of the form 
	$$t\mapsto \lambda t$$ 
for some accumulation point $\lambda$ of $\{\lambda_j\}\subset \R$.
We  may thus find  a subsequence of $$\{t\mapsto \lambda \inv\lambda_j t+ v_j\}$$ converging to the identity in $\Aff(p)$.  It follows that $\Aff(p)$ is not discrete.  In particular, for every $p\in G$ the group $\Aff(p)$ contains a one-parameter subgroup of $\AffR$.

\def\Cc{\mathcal C}
For $p\in X$ denote by $\Cc_p \colon \R \to \R$ the map $$\Cc_p = \I_{F(p)}\circ DF_p \circ \I\inv _p$$ where $\I_p$ denotes the trivialization \eqref{eq:trivialization}.
 Noting that $(\Cc_p)_*\bar\mu_p \simeq \bar\mu_{F(p)}$ we have that $$\Aff(F(p)) = \Cc_p\Aff(p) \Cc_p\inv.$$
Let $\Aff_0(p)\subset \Aff(p)$ denote the identity component of $\Aff(p)$. 
Then $\A_0(F(p)) $ is isomorphic to $\A_0(p)$ for \ae $p\in X$.  
Since $\mu(G)>0$, it follows by ergodicity that $\A_0(p)$ contains a one-parameter subgroup for \ae $p\in X$.  

The one-parameter subgroups of $\AffR$ are either pure translations or are conjugate to scaling.  We show that $\Aff(p)$ contains the group of translations for \ae $p\in X$.  
Suppose for purposes of contradiction that $\Aff_0(p)$ were conjugate to scaling for a positive measure set of $p\in X$. By ergodicity, it  follows that $\Aff_0(p)$ is conjugate to scaling for \ae $p\in X$.  
For such $p$, there are $t_0\in \R, \gamma\in \R_+$ with 
$$\Aff_0(p) = \{t\mapsto t_0 + \gamma^s(t-t_0)\mid s\in \R\}.$$
In particular, for such $p$ the action of $\Aff_0(p)$ on $\R$  contains a unique fixed point $t_0(p)$. 

For $p\in G$ the fixed point $t_0(p)$ is non-zero since, as observed above, there are $\psi\in \Aff(p)$ arbitrarily close to the identity with $\psi(0) \neq 0$.  
Furthermore, writing $\psi\colon t \mapsto t_0(p) + \gamma^s(t-t_0(p))$ we have 
$$\Cc _p \circ \psi\circ \Cc_p\inv = \pm\|\restrict{ DF^n}{E^u(p)}\|t_0(p) + \gamma^s\left(t-\pm\|\restrict{ DF^n}{E^u(p)}\|t_0(p)\right)$$ 
where the sign depends on whether or not $\Cc _p\colon \R \to \R$ preserves orientation.  
It follows for $p\in G$ that $|t_0(F^n(p)) |= \|\restrict{ DF^n}{E^u(p)}\|  \ |t_0(p)|$ becomes arbitrarily large, contradicting Poincar\'e recurrence.

Therefore, for almost every $p\in X$, the group $\Aff(p) $ contains the group of translations.  
We finish the proof by showing that for such $p$, the measure $\bar \mu_p$ is \emph{invariant} under the group of translations. 
 For $s\in \R$ define $T_s\colon \R \to \R$ by $T_s\colon t \mapsto t+s$
 and 
define $c_p\colon \R \to \R$ by 
$c_p(s) = \bar \mu _p([-s-1,-s+1])$.  Then 
$$\frac{d(T_s)_* \bar\mu_p} { d\bar \mu_p} = c_p(s) .$$  Since (by the positive entropy hypothesis) $\bar\mu_p$ contains no atoms, we have that $c_p\colon \R \to \R$ is continuous.

Note that 
$$\Cc_p \circ T_s \circ \Cc_p\inv = T_{ \pm \|\restrict{ DF}{E^u(p)}\| s}$$ and for $n\in \Z$
\begin{equation}c_p(s) = c_{F^n(p)}\left( \pm  \|\restrict{ DF^n}{E^u(p)}\| s \right)\label{eq:bad}\end{equation} where the signs depend on whether or not $DF$ or $DF^n$ preserves the orientation on $\E^u.$
Define the set $$B_{r,\epsilon}:= \{p\in X: |c_p(t)-1|<\epsilon \text{  for all } |t|<r\}.$$  For each $\epsilon>0$ pick $r$ so that $\mu(B_{r,\epsilon})>0$.  Applying \eqref{eq:bad} for $n\to -\infty$, we violate Poincar\'e recurrence with respect to the set $B_{r, \epsilon}$ unless  $|c_p(t)-1|<\epsilon$ for all $t$ and \ae $p\in X$.  
Taking $\epsilon\to 0$  shows that $c_p(s)= 1$ for all $s\in \R $  and \ae $p\in X$ completing the proof of the lemma.
\end{proof}

Theorem \ref{thm:skewproduct} now follows as an immediate corollary of Lemmas \ref{lem:main} and  \ref{lem:translation1}.

\section{Proof of Lemma \ref{lem:main}: Preparatory Lemma}\label{sec:lemmas}
We begin with a number of constructions and technical lemmas that will be used in the proof of Theorem \ref{thm:main2}.
For Sections \ref{sec:lemmas} and \ref{sec:lemmaproof} we write $$\{\hat \nu_\xi\}_{\xi \in \Omega}$$ for the family of conditional probabilities  induced by (a partition into atoms of) $\hat \Fol$.  
We also write $X_0$ for the full $\mu$-measure, $F$-invariant subset of $\Omega_0\times M$ where all propositions from Section \ref{sec:BG} hold and such that 
the stable and unstable manifolds, Lyapunov norms, affine parameters, and the trivialization $\I$ are defined.  We further assume that for $p=(\xi,x) \in X_0$ the  measures $\mu_\xi$, $\mu^u_p$, $\mu^s_p$, and  $\bar\mu_p$ are defined, non-atomic, and satisfy
$F_*\mu_\xi = \mu_{\theta(\xi)}$, 
$F_*\mu^{u/s}_p \simeq \mu^{u/s}_{F(p)}$, and $\I_{F(p)}\circ DF_p \circ \I_p\inv(  \bar\mu_p) \simeq \bar \mu_{F(p)}$.  Finally, we assume for $p\in X_0$ that $\bar \mu_p$ contains $0$ in its support.

\subsection{Dichotomy for invariant subspaces}
We establish the following dichotomy for $DF$-invariant subbundles of $TX$. 
Let $F\colon X\to X$  and $\mu$ be as in Theorem \ref{thm:skewproductABS}.  
\def\V{\mathcal V}
Consider a $\mu$-measurable line field $\V\subset TX$.  Write $V_\xi(x)\subset T_xM$ for the family of subspaces with $$\V(\xi,x) = (\xi, (x,  V_\xi(x)) ).$$   
The measurability of $\V$ with respect to a sub-$\sigma$-algebra of $X$ is the measurability of the function $(\xi,x)\mapsto V_\xi(x)$ with the standard Borel structure on $TM$.
We say  $\V$ is \emph{$DF$-invariant} if for $\mu$-\ae $(\xi,x)\in X$
\[DF_{(\xi,x)} \V(\xi,x) = \V(F(\xi,x))\text {\ or \ }   D_xf_\xi V_\xi(x) = V_{\theta(\xi)}( f_\xi(x)).\]
		
Recall that 
$\F$  in Theorem \ref{thm:skewproductABS} is an increasing sub-$\sigma$-algebras; that is,   $F(\F) \subset \F$ 
We write $\F_{\infty}$ 
for the smallest $\sigma$-algebra containing $\bigcup_{n\ge 0}F^{-n}(\F)$.  We  similarly define  $\hat \F_{\infty}$.  
(We remark that in the case that $\hat \Fol$ is the $\sigma$-algebra of local unstable sets in Section \ref{sec:skewreint}, $\hat\Fol_\infty$ and $\Fol_\infty$ are, respectively,  the completions of the Borel algebras on $\Sigma$ and $\Sigma\times M$.) 
		
\begin{lemma}\label{lem:VFdichot}
		Let $\mu$ and $\scrF$ be as in Theorem \ref{thm:skewproductABS}.  Then 
	\begin{enumerate}[label=({\arabic*}), font=\normalfont]
		\item \label{dich1}  the line field $(\xi,x)\mapsto E_\xi^s(x)$ is $\F_\infty$-measurable;
		\item \label{dich2}  
		for any $DF$-invariant, $\F_{\infty}$-measurable line field   $\V\subset TX$ either 
${(\xi,x) \mapsto V_\xi(x)}$  is $\F$-measurable, or 
			\begin{equation}\label{eq:VFdichotomy}  \text{for $\nu$-\ae $\xi$, $\hat \nu_\xi$-\ae $\eta$, and $\mu_\xi$-\ae $x$,  \( V_\xi(x) \neq V_\eta(x).\)}\end{equation}
		\end{enumerate}
\end{lemma}

\begin{proof}
To see \ref{dich1} we recall that $\xi\to \cocycle    [\xi][-n]$ is $\hat\Fol$-measurable for all $n\ge1$.  Then 
$$  \xi\mapsto \cocycle = \left(\cocycle[\theta^n(\xi)][{-n}]\right) \inv$$
is $\theta^{-n} (\hat \F)$-measurable.  It follows that $\xi\to \cocycle $ is $\hat \F_\infty$-measurable for all $n\ge 0$.  
Since $E^s_\xi(x)$ depends only on $\cocycle$ for $n\ge 0$, we have  $$(\xi,x)\mapsto E^s_\xi(x) = \left\{v\in T_xM \mid  \lim _{n\to \infty} \dfrac 1 n |D\cocycle  (v)| <0\right\}$$ is $\F_\infty$-measurable.  

To prove \ref{dich2} 
we introduce the following objects.
\def\Qpart{\mathcal Q}
\begin{itemize}
\item Let $\P$ denote the measurable partition of $X$ into level sets of $(\xi,x) \mapsto V_\xi(x)$.
\item Let $\Qpart$ denote a measurable partition of $X$ into atoms of $\Fol$.  
\item Let $\mu^\Qpart_{(\xi,x)}$ denote a family of conditional probabilities of $\mu$ with respect to $\Qpart$
\end{itemize}

First observe that the $\hat \F$-measurability of $\{\mu_\xi\}$ implies that  
conditional measures $\mu^\Qpart_{(\xi,x)}$ are lifts of the associated conditional measures on $\Omega$:
\[\ d \mu^\Qpart_{(\xi,x)}(\cdot, x ) = \ d \nuhat_\xi(\cdot)\]
and that we can exchange quantifiers in the second conclusion of the lemma: \eqref{eq:VFdichotomy} holds if and only if 
		\begin{align}\label{eq:olive} 
\text{for $\nu$-a.e.\ $\xi$, $\mu_\xi$-a.e.\ $x$, and  $\hat \nu_\xi$-a.e.\ $\eta$, 
\(V_\xi(x) \neq V_\eta(x).\)}
\end{align}

We assume \eqref{eq:olive} fails; that is, we assume 
\begin{equation}\begin{aligned}\label{eq:treeness}
	\mu\bigg\{(\xi,x) \mid \mu^\Qpart_{(\xi,x)}& (\P(\xi,x))  > 0\bigg\}= \mu\bigg\{(\xi,x) \mid \hat \nu_\xi\big\{ \eta \mid V_\xi(x) = V_\eta(x) \big\} > 0\bigg\}>0.  
\end{aligned}\end{equation}
From \eqref{eq:treeness} we will deduce $\F$-measurability of $(\xi,x)\mapsto V_\xi(x)$.

\def\meashere{\mu _{(\xi,x)}^\Qpart}
\def\measheren{\mu _{(\xi,x)}^{\Qpart_n}}

Let 
	\[\Fol_{n}: = F^{-n} (\F)\]
and write  $\Qpart_n= F^{-n}(\Qpart)$ for the partition of $X$ into atoms of $\F_n$ with corresponding family of conditional measure $\{\measheren\}$.  

For each $(\xi,x)\in X$ define 
\[ \Phi_n(\xi,x) := 
\mu _{(\xi,x)}^{\Qpart_n}(\P(\xi, x)).
\]
  We have almost surely equivalent expressions
\[ \Phi_n(\xi,x) =  \Ex_{\mu _{(\xi,x)}^\Qpart}(1_{\P (\xi, x)}(\cdot)\mid \F_{n})(\xi,x) = \Ex_{\nuhat_\xi}(1_{\P (\xi, x)}(\cdot, x)\mid \hat \F_{n})(\xi). \] 
Consider any $(\xi,x)$ with $ \mu^\Qpart_{(\xi,x)} (\P(\xi,x))>0$ and such that  $\V$ is $\Fol_\infty$-measurable $\mod \mu^\Qpart_{(\xi,x)} $. 
For $\eta\in \Omega$ define $$\Psi_n(\eta) := 
 \Ex_{\nuhat_\xi}(1_{\P (\xi, x)}(\cdot, x)\mid \hat \F_{n})(\eta)
.$$
Then $\Psi_n(\eta) $ is a martingale (with filtration $\hat \F_n$ on the measure space $(\Omega, \B_\Omega, \nuhat_\xi)$) 
whence (using  the $\Fol_\infty$-measurability of $\V$)
$$\Psi_n(\eta)\to \Ex_{\nuhat_\xi}(1_{\P (\xi, x)}(\cdot, x)\mid \hat \F_{\infty})(\eta) = 1_{\P (\xi, x)}(\eta,x)$$ $ \hat\nu_\xi$-\as as $n\to \infty$.  
In particular, for $\meashere$-\ae $(\eta,x) \in \P(\xi,x)$
$$\Phi_n(\eta,x) \to 1$$ as $ n\to \infty.$
It follows from  \eqref{eq:treeness}  that 
\begin{equation}\label{eq:downlow}\mu\left\{ (\xi,x) \in X \mid \Phi_n(\xi,x) \mapsto 1 \text{\ as\ } n\to \infty\right\}>0.\end{equation}

The $\F$-measurability of $(\xi,x)\mapsto V_\xi(x)$ is equivalent to the assertion that 
\[\mu\{(\xi,x)\mid \Phi _0(\xi,x) = 1\} = 1.\]  
We claim that $\Phi_0(F^n(\xi, x)) = \Phi_n(\xi,x).$
Indeed, we have $F^n( \Q_n(\xi,x)) =  \Q(F^n(\xi,x))$ whence 
$$F^n_*(\mu _{(\xi,x)}^{\Qpart_n}) = \mu _{F^n(\xi,x)}^{\Qpart}$$ and
$$F^n(\P(\xi,n)\cap \Q_n(\xi,x)) = \P(F^n(\xi,n))\cap \Q(\xi,x))$$
almost surely by the  $\F_n$-measurability of  $(\xi,x) \mapsto D_x\cocycle$  and  $DF$-invariance of $\V$. 
The ergodicity and $F$-invariance of $\mu$, the identity $\Phi_0(F^n(\xi, x)) = \Phi_n(\xi,x)$, and \eqref{eq:downlow} together imply that $\Phi_0\equiv 1$ on a set of full measure completing the proof.
\end{proof}

\subsection{Definition and bounds on the measure of recurrence sets.}

\def\rec{\mathcal R}
Consider $\phi\colon X\to \R$  an integrable function.  We write $\rec(\phi)$ for the set of points
 \[\rec(\phi) := \left\{q\in X\mid \lim _{n\to \pm \infty}\dfrac{1}{|n|+1}\sum_{k=0}^{n} \phi(F^k(q)) = \int \phi \ d \mu\right \}.\]
For a measurable set $B\subset (X)$ write \[\rec (B) := \rec(1_B).\]
Note that for any integrable $\phi$, the pointwise ergodic theorem implies  $\mu(\rec(\phi)) = 1$.   

For a subset $A\subset \N$, or more generally $A\subset \Z$, we define its
\begin{enumerate}
\item  \emph{upper density} \(\displaystyle \overline d(A):= \limsup_{t\to \infty} \dfrac {\#(A\cap [0, t])}{t},\)
\item \emph{lower density} \(\displaystyle \underline d(A):= \liminf_{t\to \infty} \dfrac {\#(A\cap [0, t])}{t}\), and 
\item \emph{density} \(\displaystyle d(A):= \lim_{t\to \infty} \dfrac {\#(A\cap [0,t])}{t}\) whenever the limit exists.
\end{enumerate}
We observe that each limit above (if defined) is independent of taking $t\to \infty$ in $\Z$ or $\R$ and is not changed if finitely many elements of $A$ are omitted.  
For measurable $B\subset X$ we have
$$x\in \rec (B) \iff  d(\{n\in \N \mid f^n(x) \in B\})=d(\{n\in \N \mid f^{-n}(x) \in B\})=  \mu(B).$$

\def\dens{\mathcal D}
\subsubsection{Sets with good recurrence along the stable foliation} 
We have the following lemma.
\begin{lemma}\label{lem:rec1}
For any  $\rho>0$ and   $\delta>0$ there exists an $r_0>0$ with the following property: For any measurable $B\subset X$ with $\mu(B)>1-\delta$ 
we write
\begin{align}\label{eq:1stdensity}
\dens^1_n(x,\xi, B):= \inf_{ 0<r\le r_0}\left\{ \dfrac{\scond{ \locstab x \xi    \cap F^{-n}(B)}}{\scond {\locstab x \xi}}\right\}.\end{align}
Then for $\mu$-a.e. $(\xi,x)$
	\begin{align}\label{eq:conclusion1}\underbar d\left(\left\{
	n\in \N \mid \dens^1_n(x,\xi,B) >1-\rho
	\right\} \right) > 1-\delta.\end{align}
\end{lemma}
\begin{proof}
For each   $r'$ define 
\[G_{r'} = \{
q\in X\mid  \inf_{ 0<r\le r'} \left\{\dfrac{\scond [ q]{ \locstabp q   \cap B}}{\scond  [q] {\locstabp q }} \right\}>1-\rho.
\}\]
By the density theorem, for $\mu$-\ae $q\in B$, $$\lim_{r\to 0} \dfrac{\scond [ q]{ \locstabp q   \cap B}}{\scond  [q] {\locstabp q }} = 1.$$  We may thus find $r_0>0$ sufficiently small so that $\mu(G_{r_0})>1-\delta$.  

We prove the lemma for  $p= (\xi,x)  \in \rec(G_{r_0})$.  
By Proposition \ref{prop:growth}, there is an $N$ such that for all $n\ge N$ and $r\le r_0$
$$F^n(\locstabp p) \subset \locstabp [r_0]{F^n(p)}.$$
Then for $n\ge N$
$$F^n(\xi,x) \in G_{r_0} \iff  \dens^1_n(x,\xi,B) >1-\rho. $$
\eqref{eq:conclusion1} follows since $(\xi,x)  \in \rec(G_{r_0})$.
\end{proof}


We need a similar but somewhat more complicated  
lemma describing the density of recurrence points along \emph{non-contracting},  \emph{non-invariant} foliations.  For $\eta$ and $\xi$ in the same atom of $\hat \Fol$, we measure the density of preimages of a set $B$ in the fiber over  $\eta$ along the stable foliations defined by $\xi$.  

To avoid confusion in notation, we write $\hol_{\eta,\xi}$ for the trivial identification map 
\begin{align*} 
\hol_{\eta,\xi}\colon M_\eta\to M_\xi, \quad 
\hol_{\eta,\xi}\colon (\eta,y) \mapsto (\xi, y).
\end{align*}
Recall that  $\xi\to \mu_\xi$ is assumed $\hat \Fol$-measurable whence,  for $\nu$-\ae $\xi$ and $\hat \nu_\xi$-\ae $\eta\in \Omega$, 
$$( \hol_{\eta,\xi})_*\mu_\eta = \mu_\xi.$$

		\begin{lemma}\label{lem:rec2}
		For all $\delta>0$ and $\rho>0$  we have the following: 
		
		For any measurable $B\subset X$ with $\mu(B)> 1-\dfrac \rho 3 \delta$, 
		for $\nu$-a.e.\ $\xi$ and $\hat \nu_\xi$-\ae $\eta$  there is a $\mu_\xi$-measurable function $r_0\colon M_\xi \to (0,\infty)$ such that, writing  
		\begin{align} \label{eq:2nddensity}\dens^2_n(x,\xi,\eta,B) :=
		 \inf_{ 0<r\le r_0(x)}\left\{ \dfrac
		 {\scond { \locstab x \xi  \cap \hol_{\eta,\xi}\Big(F^{-n}(B)\cap M_\eta\Big)}}
		 {\scond {\locstab x \xi }} 
		 \right\},\end{align}
we have 
\begin{align}\label{eq:conclusion2}
		\mu_\xi\Big\{
		x\in M_\xi\mid  
		\bar d\Big(\left\{n\in \N_0\mid 
			\dens^2_n(x,\xi,\eta,B)
		>1-\rho
		\right\}\Big)>1-\delta
		\Big\}>\dfrac{1}3.
\end{align}
		\end{lemma}
\def\HH{\hol_{\eta,\xi}}

\begin{proof}
Write $\hat \delta := \dfrac \rho 3 \delta$.  
We note that, for a given $B$ and the $\nu$-measurable function $\eta\mapsto \mu_\eta(B)$, for $\nu$-a.e. $\xi$ and $\hat \nu_\xi$-\ae $\eta$
\begin{equation}\begin{aligned}\label{boowhoo}
 \lim _{N\to \infty} \dfrac{1}N\sum_{k=0}^{N-1}\mu_{\xi }\big(\HH(F^{-k}(B)) \big)
  &= \lim _{N\to \infty} \dfrac{1}N\sum_{k=0}^{N-1}\mu_{\eta }(F^{-k}(B))\\
  &= \lim _{N\to \infty} \dfrac{1}N\sum_{k=0}^{N-1}\mu_{\sigma^k(\eta) }(B)\\
  &> 1-\hat \delta.\end{aligned}\end{equation}
We prove the lemma for any $\xi$ such that 
\begin{enumerate}
\item  for $\hat \nu_\xi$-a.e.\ $\eta\in \Omega$, we have $\mu_\xi= \mu_\eta$;
\item $\hat \nu_\xi$-a.e.\ $\eta\in \Omega$ satisfies \eqref{boowhoo};
\item {$\mu_\xi(X_0) = 1$}. 
\end{enumerate} 

Let $\P$ be a $\mu_\xi$-measurable partition of $M_\xi$ subordinate to the partition of $M_\xi$ into stable manifolds $\{\stab x \xi\}$. \def\M{\mathcal M}
We choose the function $r_0= r_0(x)$ in the lemma so that $\locstab[r_0]x \xi  \subset \P(x)$ for $\mu_\xi$-\ae $x\in M_\xi$.  
Write $\{\td \mu^\P_x\}_{x\in M_\xi}$ for a family of conditional measures of $\mu_\xi$ with respect to the partition $\P$.  Up to normalization, $\td \mu^\P_x$ is \ae the restriction of $\mu^{s}_{(\xi,x)}$ to  $\P(x) \subset \stab x \xi$.  

For $x\in M_\xi$ and $k\in \N_0$ define the functions
	\[\M^k(x) := \inf_{0<r\le r_0(x)}\dfrac{\td \mu^\P_x\left( \locstab x \xi \cap \HH\big(F^{-k}(B) \big) \right)}{\td \mu^\P_x\left( \locstab x \xi \right) } \]
and 
\[\widehat \M^k(x) := \sup_{0<r\le r_0(x)}\dfrac{\td \mu^\P_x\left( \locstab x \xi \sm \HH\big(F^{-k}(B) \big) \right)}{\td \mu^\P_x\left( \locstab x \xi\right) } .\]
Clearly  $\M^k(x) = 1- \widehat \M^k(x)$.  
Note that the constant in the Besicovitch covering lemma for $\R$ is $2.$  Restricting the maximal function $\widehat M^k$ 
to the probability space $\P(y)\subset (\stab x \xi , \td \mu^\P_y) \cong (\R,  \td \mu^\P_y)$,  for every $k\in \N_0$ 
we have the maximal inequality
\[
\td \mu^\P_y\left(\left \{
x\in \P(y)\mid \widehat \M^k(x) >\rho
\right\}\right)\le \dfrac{2}{\rho}\ \td \mu^\P_y\left(\P(y) \sm \HH\big( F^{-k}(B)\big)\right).
\]
It then follows that  for each $k$ and $\mu_\xi$-\ae $y$, 
\[\td \mu^\P_y\left(\left \{
x\in \P(y)\mid \M^k(x) >1- \rho
\right\}\right)\ge 1- \dfrac{2}{\rho}\td \mu^\P_y\left(\P(y) \sm\HH\big( F^{-k}(B)\big) \right).
\]

The lemma follows if we show that
\[\mu_\xi \left \{x\mid \bar d( \{k\in \N_0 \mid \M^k(x)>1-\rho\})>1- \delta\right\}>\dfrac 1 3.\] 
Set 
\[G:=\{(x,k) \in M_\xi\times\N_0\mid \M^k(x) >1-\rho\}.\]
From \eqref {boowhoo} 
we have 
\begin{align*}
 \limsup_{N\to \infty} 
 \dfrac{1}{N} \sum_{k = 0}^{N-1}  &\int_{M_\xi}   1_{G}(x,k)  \ d\mu_\xi(x)
 \\&=  \limsup_{N\to \infty} 
\dfrac{1}{N} \sum_{k = 0}^{N-1}   \int_{M_\xi}  \td \mu^\P_y\left(\left \{
x\in \P(y)\mid \M^k(x) >1- \rho
\right\}\right)  \ d\mu_\xi(y)
 \\&\ge \limsup_{N\to \infty} 
  \dfrac{1}{N}\sum_{k = 0}^{N-1}\int_{M_\xi}  \left( 1- \dfrac{2}{\rho}\td \mu^\P_y\left(\P(y) \sm \HH\big( F^{-k}(B)\big)\right)\right ) \ d \mu_\xi(y) 
  \\&=  1 - \lim_{N\to \infty}   
  \dfrac {2}{\rho}\dfrac{1}{N}\sum_{k = 0}^{N-1}  \mu_\xi \left(M_\xi\sm\HH\big({F^{-k}(B)}\big)\right)
  \\&>  1 - \dfrac 2\rho \hat \delta = 1 - \dfrac 2 3 \delta. 
\end{align*}

Write \[\psi (x) : = \limsup_{N\to \infty} \dfrac 1 N \sum_{k = 0}^{N-1}    1_{G}(x,k)   =\bar d( \{k\in \N_0 \mid \M^k(x)>1-\rho\}).  \] 
From Fatou's lemma we have
\[ 
\int\psi (x) \ d \mu_\xi(x)
\ge\limsup_{N\to \infty} \dfrac{1}{N}\sum_{k = 0}^{N-1}      \int_{M_\xi}  1_{G}(x,k)    \ d\mu_\xi(x)
> 1 -  \dfrac 2 3 \delta  .\]
Then 
\begin{align*}1 -  \dfrac 2 3 \delta  
&<  \int_{M_\xi }\psi \ d \mu_\xi\\
&\le \mu_\xi\left(\{x\mid \psi(x)\ge1 -   \delta\} \right)  +(1 -  \delta)\left(1-  \mu_\xi\left(\{x\mid \psi(x)\ge1 -  \delta\} \right) \right) \\
 &= \delta \mu_\xi\left(\{x\mid \psi(x)\ge1 -  \delta\} \right)  +1 -   \delta \end{align*}
whence
\[ \mu_\xi(\{x\mid \psi(x)\ge 1 - \delta \} )> \dfrac1 3\]
concluding the proof.\end{proof}

				\subsection{Quasi-isometric estimates for stopping times} 

Consider fixed  $x\in M$ and $ \xi,\eta\in \Omega$ such that { $(\xi,x)\in X_0$ and $(\eta,x)\in X_0$.}  For any $r>0$ define \emph{stopping times} 
$$ \tau_{1,r}= \tau_{1,r, x, \xi}\colon \N_0\to \Z, \quad \quad   \tau_{2,r}= \tau_{2,r, x, \xi, \eta}\colon \N_0\to \Z$$ by 
\begin{align}
	\label{eq:tau1}\tau_{1,r}(m) &: = \sup\{n\in \Z : \lyap{\restrict{DF^n}{\Eu x \xi )}}_{\eps}  
							\ 	\lyap{\restrict{DF^{-m}} {\Es x\xi)}}_{\eps} \inv\}< r;		\\	
\label{eq:tau2}\tau_{2,r}(m) &: = \sup\{n\in \Z :\lyap{\restrict{DF^n}{\Eu x \eta}}_{\eps} 
							\ 	 \lyap{\restrict{DF^{-m}} {\Es x \xi)}}_{\eps} \inv\}< r.
\end{align}
We remark that each $\tau_{j,r}$ is increasing and takes only finitely many non-positive values.  

			\begin{lemma}
			The maps $\tau_{j,r}\colon \N_0\to \Z$ are quasi-isometric embeddings  with constants uniform in $r$.  That is, setting
			\begin{align}
			L &:= \max\left\{\dfrac{-\lambda ^s+ \eps}{\lambda^u-\eps}, \dfrac{\lambda^u+\eps}{-\lambda ^s- \eps}\right\}\\
			a &:= \lambda^u + \eps,\end{align}
			 for all $m,\ell\in \N_0, r>0$ and $ j\in \{1,2\}$ 
			\begin{align}\dfrac 1 L \ell - a\le \tau_{j,r}(m+\ell) - \tau_{j,r} (m) \le L\ell + a\label{eq:QIdefn} \end{align}
			\end{lemma}
\begin{proof}
We prove the lemma for $\tau:= \tau_{2,r}$; the proof and resulting estimates for $\tau_{1,r}$ are identical.  
Let $\kappa= \exp(\lambda^u+\eps)$.  
By definition we have 
\begin{equation}\begin{aligned}\label{eq:QI}
\kappa\inv r \le 
&\lyap{ \restrict{DF^{\tau(m+\ell)}}{\Eu x \eta}}_{\eps} \cdot  \   
\lyap{\restrict{DF^{-m-\ell\phantom{)}}} {\Es x \xi}}\inv_{\eps}  \\ 
& \quad \quad\cdot \  \lyap{\restrict{DF^{\tau(m)}}{\Eu x \eta}}_{\eps} \cdot \ 
  \lyap{\restrict{DF^{\tau(m)}}{\Eu x \eta}}\inv_{\eps} \ 
  \le r.\end{aligned}\end{equation}
We bound the product of the middle terms of \eqref{eq:QI} by
\begin{align*} &\exp((\lambda^s - \eps)\ell) \kappa \inv r
\\& \le \lyap{\restrict{DF^{-m-\ell}} {\Es x \xi}}\inv_{\eps} \cdot \  \lyap{\restrict{DF^{\tau(m)}}{\Eu x \eta}} _{\eps}\ 
\\ &\le \exp((\lambda^s + \eps)\ell) r\end{align*}
and the product of outermost terms of \eqref{eq:QI}  by
\begin{align*}
&\exp((\lambda^u-\eps)(\tau(m+\ell) - \tau(m))) \\
&\le  \lyap{\restrict{DF^{\tau(m+\ell)}}{\Eu x \eta}} _{\eps}\cdot \ \lyap{\restrict{DF^{\tau(m)}}{\Eu x \eta}} \inv _{\eps}
\\ &\le \exp((\lambda^u+\eps)(\tau(m+\ell) - \tau(m)) )
.\end{align*}
Reassembling, \eqref{eq:QI} we  have 
\[
\exp\big((\lambda^u-\eps)(\tau(m+\ell) - \tau(m))\big)
\exp((\lambda^s - \eps)\ell) \kappa \inv r\le r\]
and 
\[\exp\big((\lambda^u+\eps)(\tau(m+\ell) - \tau(m))\big) 
\exp((\lambda^s + \eps)\ell) r\ge \kappa\inv r\] 
hence 
\[
\dfrac{-\lambda ^s- \eps}{\lambda^u+\eps} \ell - \log \kappa
\le 
\tau(m+ \ell) - \tau(m)
\le 
\dfrac{-\lambda ^s+ \eps}{\lambda^u-\eps} \ell + \log \kappa. \qedhere\]
 \end{proof}

\begin{lemma}
Let $\tau\colon\N_0\to \Z $ be a quasi-isometric embedding  with constants $L$ and $a$ satisfying \eqref{eq:QIdefn}.  
For any $G\subset \N_0$ we obtain the following bounds on the densities of the image and preimage of $G$ under  $\tau$: 
\begin{enumerate}[label={\arabic*}), ref={\arabic*})]
\item $\dfrac{1}{L^2 a} \underbar d(G) \le \underbar d(\tau(G))$
\item $\dfrac{1}{L^2 a} \bar d(G) \le \bar d(\tau(G))$ \label{item:2here}
\item if $\underbar d(G) \ge1-\delta$ then $\underbar d (\tau\inv (G))\ge 1- (L^2a)\ \delta$.  
\end{enumerate}
\end{lemma}

\begin{proof} 
We first note that if $\tau(n) = \tau (n+k)$, $k\ge 0$, then 
\[ 0 = \tau(n+k) - \tau(n) \ge \dfrac 1 L k - a\]
hence $k\le La$.  In particular, the map $\tau\colon \N_0\to \Z$ 
 is at worst $(La)$-to-one.  
Secondly, 
we have the inclusion
 \[\tau([0, t]\cap \N_0)\subset [\tau(0)-a, Lt + a +\tau(0)].\] 
We thus obtain bounds
\begin{align*}
 \dfrac{\#([0,t]\cap G)}{t}
	&\le    (La)\  \dfrac{\#\big([\tau(0)-a,Lt+a +\tau(0)]\cap \tau (G)\big)}{t}\\
	&=  (L^2a)\  \dfrac{\#\big([\tau(0)-a,Lt+a +\tau(0)]\cap \tau (G)\big)}{Lt}.
\end{align*}
Taking $\liminf_{t\to \infty}$ and $\limsup_{t\to \infty}$ yields the first two bounds.  
 
For the final bound, let $G^c:= \Z\sm G$.  Then $\tau\inv(G)$ and $\tau\inv (G^c)$ partition $\N_0$.  
We have 
\[\bar d(G^c) :=  \limsup_{N\to \infty}\left(1-  \dfrac {\#([0,N]\cap  G)}{N} \right)= 
1 -  \liminf_{N\to \infty}  \dfrac {\#([0,N]\cap   G)}{N} \le \delta.\]
hence by \ref{item:2here} 
$$\bar d (\tau\inv (G^c)) {\le(L^2a) \ \bar d (\tau (\tau\inv(G^c))) } \le  (L^2a) \ \bar d (G^c)\le (L^2 a) \delta.$$
We then have 
\[\underbar d (\tau\inv (G) )\ge 1 - \bar d (\tau\inv (G^c)) \ge 1- (L^2 a) \delta. \qedhere\]
\end{proof}

\def\M{\mathscr M}
As a corollary, we obtain the following.  
\begin{lemma}\label{lem:jointreturns}
Let $\good_i,  \M \subset \N_0$ satisfy
\begin{enumerate}
\item $\bar d (\good_1) \ge1-\delta$
\item $\underbar d (\good_i) \ge1-\delta$ for $2\le i \le 4$
\item $\underbar d (\M) \ge1-\delta$
\end{enumerate} and let 
$\tau_1, \tau_2\colon \N_0\to \Z$ be a quasi-isometric embedding with constants $L$ and $a$ satisfying \eqref{eq:QIdefn}.  
Then \[\bar d\bigg(
		\good_1\cap \good_3 \cap  \tau_2\Big(
				\M\cap \tau_{1}\inv\big(\good_2\big)\cap \tau_{1}\inv\big(\good_4\big)
		\Big)
\bigg)\ge\dfrac 1{L^2a} \big(1-(4L^2a +1) \delta\big).\]
\end{lemma}

\begin{proof}
For $A, B\subset \N_0$ we have 
 bounds 
$\underbar d(A) =1- \bar d(A^c)$, 
$\underbar d(A\cap B) \ge 
\underbar d(A) - \bar d(B^c)$, and  
$\bar d(A\cap B)\ge \bar d(A) - \bar d(B^c)$.
Thus $$\underbar d\Big(
				\M\cap \tau_{1}\inv\big(\good_2\big)\cap \tau_{1}\inv\big(\good_4\big)
		\Big)\ge 1- 2(L^2a) \delta - \delta, $$
 $$\underbar d \bigg(\good_3 \cap  \tau_2\Big(
				\M\cap \tau_{1}\inv\big(\good_2\big)\cap \tau_{1}\inv\big(\good_4\big)
		\Big)\bigg)\ge \dfrac 1 {L^2 a} \big( 1- 2(L^2a) \delta - \delta\big) - \delta,$$
and 
\begin{align*}\bar d\bigg(
		\good_1\cap \good_3 \cap   \tau_2\Big(
				& \M\cap \tau_{1}\inv\big(\good_2\big)\cap \tau_{1}\inv\big(\good_4\big)
		\Big)
\bigg)
\\ &
\ge
 \dfrac 1 {L^2 a} \big( 1- 2(L^2a) \delta - \delta\big) - 2\delta 
.\qedhere\end{align*}
\end{proof}

\subsection{Bounds on measures of accumulation sets}\label{sec:accset}
Recall that we assume $\Omega$ to be Polish, whence  $X$ is  second countable.  Let $\U$ be a countable basis for the topology on $X$ and let $\U^*$ be the set of all finite unions of elements from $\U$.  We have that  $\U^*$ is countable whence $$\mu\left(\bigcap_{O\in U^*} \rec(O )\right)= 1.$$

\begin{lemma}\label{lem:measureofaccumulation}
Let $\good\subset \N_0$ have upper density $\bar d(\good) \ge \gamma$.  Let $q\in \bigcap_{O\in U^*} \rec(O )$ and assume the set \[\{F^n(q)\mid n\in \good\}\] is precompact. 
Let 
\[G:= \bigcap_{N= 1} ^\infty \overline{
						\{F^k(q) \mid k\in [N,\infty)\cap \good\}  }.\]
Then $\mu(G)\ge\gamma.$
\end{lemma}

\begin{proof}

Since $G$ is compact, for every open $U\supset G$ there is an $O\in \U^*$ with $G\subset O\subset U$.  
Furthermore, by the definition of $G$ and the precompactness of the sequence $\{F^j(q)\mid j\in \good\}$ there is some $M>0$ such  that     \(F^j(q)\in O \)  for all $j\in \good \cap [M,\infty)$.
Using that $q\in \rec(O)$ we have
  \begin{align*}
\mu (U) 	 	&\ge \mu(O)= \lim_{N\to \infty}\dfrac 1 N \sum _{j = 0} ^{N-1} 1_O(F^j(q))\\
 		& \ge \limsup_{N\to \infty}\dfrac { \#([0, N-1] \cap [M,\infty)\cap \good)}{N}\\
		& = \bar d(\good) \ge \gamma.
 \end{align*}

Since $\mu$ is a finite Borel measure on a Polish space $X$, it is  outer regular  
whence
\[ \mu(G) = \inf_{U\supset G} \mu(U)\ge \gamma.\qedhere\]
\end{proof}

\section{Proof of Lemma \ref{lem:main}}\label{sec:lemmaproof}
In this section we complete the proof Lemma \ref{lem:main}. We do this in a sequence of steps. 
\subsection{Step \step: Choice of parameters and Lusin sets}
We have $\lambda^s<0<\lambda^u$ given by the dynamics $F$ and the measure $\mu$.  Recall $\eps$ fixed in Section \ref{sec:BG} and   set
\begin{align*}
			L &:= \max\{\dfrac{-\lambda ^s+ \eps}{\lambda^u-\eps}, \dfrac{\lambda^u+\eps}{-\lambda ^s- \eps}\}\\
			a &:= \lambda^u + \eps.\end{align*}
Fix $\rho= \dfrac{1}{10}$ and choose $0<\delta<\rho$ so that 
\begin{align}\label{eq:delta0}\delta_0:=\dfrac 1{L^2a} \big(1-(4L^2a +1) \delta\big)>0.\end{align}
This will be the $\delta_0$ in Lemma \ref{lem:main}.  

We apply Lusin's theorem to all the structures developed in Section \ref{sec:BG} to find a compact subset 
$K\subset  {X_0}$ with $\mu(K)>1-\dfrac \rho 3 \delta$ so that each of the following varies continuously on $K$:
\begin{enumerate}[label=\emph{\roman*})]
 \item the stable and unstable line fields $(\xi,x) \mapsto E^s_\xi(x)$ and $(\xi,x) \mapsto E^u_\xi(x)$,
\item the choice of orientation on $\E^u$;
 \item the family of local stable and unstable manifolds 
$$(\xi,x)\mapsto \locstabM [1]  x\xi\quad \quad (\xi,x)\mapsto \locunstM [1]  x\xi$$ as $C^1$ embedded curves in $M$ parametrized by affine parameters (see \eqref{eq:defofparam}); 
\item the family $p\mapsto \bar\mu_p$;
\item the $s$- and $u$-Lyapunov norms.
\end{enumerate}
We also assume that the function $(p, z)\mapsto \rho_{p}(z)$, where $\rho_p(z)$ is as in \eqref{eq:rho}, is bounded on $\bigcup_{p\in K} \locunstp [1] p$ and that there exist $C_1, \hat \gamma, \hat r$, and $\Lambda$       satisfying Lemma \ref{lem:standardcrap} with $K\subset \Lambda$.  

Let $A_N\subset K$ be the set of points 
 $$A_N:= \left\{(\xi, x)\in K\mid \dfrac{
		\scond {\locstab[1] x \xi \sm \locstab[\frac 1 N] x \xi
		 }
		}
		{\scond {\locstab[1] x \xi  }
		}>\dfrac 1 2\right\}.
$$
Since $\mu_p^s$ is non-atomic, by letting $N\to \infty$ we may find an $N_0$ so that $\mu(A_{N_0})>1-\delta$.  Fix such an $N_0$ and set $A= A_{N_0}$.   

\subsection{Step \step: Choice of $\xi$, $\eta$, $x$}
We  select $\xi, \eta\in \Omega$ and $x\in M$ that will be fixed for the remainder.     The sets $G_\epsilon$ in Lemma \ref{lem:main} will be the set of accumulation points of $F^{k_j}(\eta,x)$ for appropriate subsequences of $(k_j)\subset \N$ satisfying good recurrence properties.  We note that we actually obtain a positive measure of such points $(\eta,x)$  but only use the existence of a single point.  
We recall the countable basis  $\U$ for the topology on $X$ from Section \ref{sec:accset} and the notation $\U^*$. 

Recall that in Lemma \ref{lem:main} we assume that $(\xi,x) \mapsto E^s_\xi(x)$ is not  $\Fol$-measurable.
By Lemma \ref{lem:VFdichot},  the $\hat \Fol$-measurability hypotheses in Theorem \ref{thm:skewproductABS}, the pointwise ergodic theorem, and Lemma \ref{lem:rec1} 
we have  for 
 $\nu$-\ae $\xi$, $\hat \nu_\xi$-\ae $\eta$, and $\mu_\xi$-\ae $x\in M$ that 
\begin{enumerate}[label=\emph{\roman*}{)}, font=\normalfont]
	\item \label{item:start} $\mu_\xi = \mu_\eta$;
	\item for all $n\ge 1$, $\cocycle [\xi] [-n] = \cocycle [\eta] [-n] $ whence  $E^u_\xi(x) = E^u_\eta(x)$, $\unstM x \xi = \unstM x \eta $, and $H^u_{\xi,x} = H^u_{\eta,x}$;
	\item $\bar\mu_{(\xi,x)} = \bar \mu_{(\eta,x)}$;
	\item for $\mu_{(\xi,x)}^s$-\ae $(\xi,y) \in \stab x\xi$ we have 
	$E^u_\xi(y) = E^u_\eta(y)$, $\unstM y \xi = \unstM y \eta $,  $H^u_{\xi,y} = H^u_{\eta,y}$, and $\bar\mu_{(\xi,y)} = \bar \mu_{(\eta,y)}$;
	\item  \label{item1:2} $\EsM x \eta \neq \EsM x \xi$;
	
	\item \label{item1:3}$(\xi,x) \in \rec (K)$ and  $(\xi,x)\in \rec(A)$;
	\item $(\eta, x)\in \rec(K)$; 
	\item \label{item1:6} $(\eta, x) \in \rec (O)$ for any $O\in \U^*$;
		
	\item \label{item1:10} for $\dens^1_n(x,\xi,K)$  defined in \eqref{eq:1stdensity}
				$$\underbar d\big(\left\{ n\in \N \mid  \dens^1_n(x,\xi,K) >1-\rho \right\} \big) > 1-\delta.$$ 	
	\end{enumerate}
Furthermore, by  Lemma \ref{lem:rec2}, for $\nu$-\ae $\xi$, $\hat \nu_\xi$-\ae $\eta$, the set of $x\in M$ so that 
\begin{enumerate}[label=\emph{\roman*}{)}, font=\normalfont, resume]
\item \label{item1:11}for $\dens^2_n(x,\xi,\eta,K)$  defined in \eqref{eq:2nddensity} 
				$$\bar d\big(\left\{n\in \N_0\mid \dens^2_n(x,\xi,\eta,K)	>1-\rho\right\}\big)>1-\delta$$
\end{enumerate}	
	has $\mu_\xi$-measure at least $\frac 1 3$.

Finally, for a positive $\nu$-measure set of $\xi$ and a positive $\hat \nu_\xi$-measure set of  $\eta$, both $\mu_\xi(K)\ge 1-\delta$ and $\mu_\eta(K)\ge 1-\delta$.  For such $\xi$ and $\eta$, 
the set of $x\in M$ satisfying
\begin{enumerate}[label=\emph{\roman*}{)}, font=\normalfont, resume]

	\item \label{item1:1} $(\xi, x) \in K$ 
	 \item \label{item1:1'} $(\eta, x)\in K$ 
	
	
	\item \label{item1:7} $(\xi, x) $ is a $\mu^s_{(\xi,x)}$-density point $K$:
	\[\lim _{r\to 0} \dfrac{\scond { \locstab x \xi \cap K}} {\scond {\locstab x \xi  }} \to 1.\]
	
	\item  \label{item1:8} $(\xi, x)$ is a $\mu^s_{(\xi,x)}$-density point $\HH(K)$:
	\[\lim _{r\to 0} \dfrac{\scond { \locstab x \xi \cap \HH( K)} } {\scond { \locstab x \xi}  } \to 1.\]
	
\end{enumerate}
has $\mu_\xi$-measure at least $1-2\delta$.   

Since $1-2\delta>2/3$ by construction, we may select a triple $\xi,\eta$, and $x$ satisfying conditions  \ref{item:start} --\ref{item1:8} above.

\subsection{Step \step: Choice of return times}
Fix $\xi, \eta,$ and $x$ from Step 1.  Consider any  $\epsilon>0$.   Recall the definitions of $ \tau_{1,\epsilon}= \tau_{1,\epsilon, x, \xi}\colon \N_0\to \Z$ and $ \tau_{2,\epsilon}= \tau_{2,\epsilon, x, \xi, \eta}\colon \N_0\to \Z$
as defined in \eqref{eq:tau1} and \eqref{eq:tau2}.  We note that all estimates in the remainder are {independent of $\epsilon$}.  

For our fixed $\xi, \eta, x$ define 
\begin{enumerate}
\item  $\good_1:= \{n\in \N _0\mid \dens^2_n(x,\xi,\eta,K)	>1-\rho\}$
\item $\good _2:= \{n\in \N_0 \mid \dens^1_n(x,\xi,K) >1-\rho\}$
\item $\good _3:= \{n\in \N_0 \mid F^{n} (\eta,x) \in K\} $
\item $\good _4:= \{n\in \N_0 \mid F^{n} (\xi,x) \in K\} $
\item $ \M := \{ n\in \N_0 \mid F^{-n} (\xi,x) \in A\}$.  
\end{enumerate}
We have that  $\good_i,$ and $\M$ satisfy the hypotheses of Lemma \ref{lem:jointreturns}.  
Define
	\[\good= \good(\xi,\eta,x):=  \tau_{2,\epsilon}\inv \big( \good_1\big)\cap  \tau_{2,\epsilon}\inv \big( \good_3\big)\cap \M \cap \tau_{1,\epsilon}\inv\big(\good_2\big)\cap \tau_{1,\epsilon}\inv\big(\good_4\big).\]
Note that for a function $g\colon Y\to Z$ and $A\subset Y, B\subset Z$ one has  $$g(g\inv( B)\cap A) = g(A) \cap B.$$ 
Thus 
$$\tau_{2,\epsilon}(\good) = 
		\good_1\cap \good_3 \cap  \tau_{2,\epsilon}\Big(
				\M\cap \tau_{1,\epsilon}\inv\big(\good_2\big)\cap \tau_{1,\epsilon}\inv\big(\good_4\big)
		\Big)
$$
and by Lemma  \ref{lem:jointreturns} \begin{align*}\bar d (\tau_{2,\epsilon}(\good))\ge \dfrac 1{L^2a} \big(1-(4L^2a +1)\delta)= \delta_0> 0.\end{align*}  
  In particular, $\good$ is infinite.  

By definition of $\good_3$ and $\good_4$, for every $j\in \good $ we have $F^{\tau_{1, \epsilon}(j)} (\xi, x)\in K$ and $F^{\tau_{2, \epsilon}(j)} (\eta, x)\in K.$

  \subsection{Step \step: Choice of $\{y_j\}$}
For each sufficiently large $j\in \good$ we select a $y_j$ satisfying Lemma \ref{lem:y} below.
  For any $j \in \N$  define $$r_j:=\|\restrict {DF^{-j}}{\Es x \xi}\|\inv.$$ 
\begin{lemma}\label{lem:y}
For every sufficiently large $j\in \good$ 
there exists  $y\in \locstabM[\rhere] x \xi$ with  $d(x,y)<\hat \gamma$ and  
\begin{enumerate}[label=\emph{\roman*}{)}, font=\normalfont]
\item \label{item6:0} $\bar \mu _{(\eta,y)} = \bar\mu_{(\xi,y)}$,
\item  \label{item6:1} $N_0\inv{ r_j}\le \| H^s_{(\xi,x)} (y)\| \le r_j$,
\item $(\xi, y) \in  K$
and $(\eta, y) \in  K$,
\item $F^{\tau_{1, \epsilon}(j)} (\xi, y)\in K$, and
\item \label{item6:5}$F^{\tau_{2, \epsilon}(j)} (\eta, y)\in K$.
\end{enumerate}
\end{lemma}
\noindent Here $\hat r$ and $\hat \gamma$ are as in Lemma \ref{lem:standardcrap}.
\begin{proof}
We note that for all but finitely many $j$, we have $r_j \le \rhere$ and $\sup \{d(x,y)\mid y\in  \locstabM[r_j] x \xi\} < \hat \gamma$.  Furthermore, \ref{item6:0} holds for almost every $y\in \locstabM[\rhere] x \xi$.  

For \ref{item6:1}--\ref{item6:5} we have that for all $j\in \good$ sufficiently large  
\begin{enumerate}[label={\alph*}{)}]
\item \label{item4:1}$ \dfrac{\scond { \locstab[r_j] x \xi \cap K}} {\scond { \locstab [r_j]x \xi} } >.9$  
\item \label{item4:2}$ \dfrac{\scond { \locstab[r_j] x \xi \cap \HH( K)}} {\scond {\locstab[r_j] x \xi} } >.9$  
\item \label{item:cc}$r_j< r_0 $ for $r_0$ given by Lemma \ref{lem:rec1} 
\item \label{item:dd}$r_j < r_0(x)$ for $r_0(x) $ given by Lemma \ref{lem:rec2} for our choice of $\xi,\eta$ and $x$.
\end{enumerate} 
From \ref{item:cc} and \ref{item:dd}, Lemmas \ref{lem:rec1} and \ref{lem:rec2}, and the definitions of $\good_1,$ and $ \good_2$ it follows that  
\begin{enumerate}[resume, label={\alph*}{)}]
\item\label{item4:3} $ \dfrac
		 {\scond { \locstab[r_j]x \xi \cap F^{-\tau_{1, \epsilon}(j)}(K)}}
		 {\scond {\locstab[r_j] x \xi }} 
		>.9$	
\item \label{item4:4}$ \dfrac
		 {\scond { \locstab [r_j]x \xi   \cap \hol_{\eta,\xi}\big(F^{-\tau_{2, \epsilon}(j)}(K)\big)}}
		 {\scond {\locstab[r_j]x  \xi}} 
		>.9$
\end{enumerate}
and from the definition of $A$ that
\begin{enumerate}[resume, label={\alph*}{)}]
\item \label{item4:5}$\dfrac{\scond { \locstab [ r_j] x  \xi  \sm  \locstab [ N_0\inv r_j ] x\xi }}{\scond {\locstab[r_j]x \xi}} >.5$.
\end{enumerate}
It follows from \ref{item4:1}, \ref{item4:2}, \ref{item4:3}, \ref{item4:4}, and \ref{item4:5} that there is a positive $\mu^s_{(\xi,x)}$-measure set of points $(\xi,y)\in \locstab[r_j] x \xi$ satisfying  \ref{item6:1}-\ref{item6:5}.  
\end{proof}
For  sufficiently large $j\in \good$, we select $y_j$ satisfying Lemma \ref{lem:y}.  Note that $r_j\to 0$, and hence $y_j \to x$,  as $j\in \good \to \infty$.

\def\ponej{p_j^1}
\def\qonej{q^1_j}
 \def\ptwoj{p^2_j} 
\def\qtwoj{q^2_j}
 \def\zj{z_j}
 \def\wj{w_j} 
 \def\oj{o_j}

 \subsection{Step \step: Bounds on distortion}  
For sufficiently large $j$ and $y_j$ satisfying Lemma \ref{lem:y} we have $(\eta, x),(\eta,y_j)\in K \subset \Lambda$ for a set $\Lambda$ satisfying Lemma \ref{lem:standardcrap}.  We then have that the intersection $$\locstabM[\rhere] {y_j} \eta \cap \locunstM[\rhere] x \xi= \locstabM[\rhere] {y_j} \eta \cap \locunstM[\rhere] x \eta$$
is a singleton. Define  $z_j \in M$ to be this point of intersection (See Figure \ref{fig:1}).   

 We have the following geometric lemma.  
\begin{lemma}\label{lem:riemGeom} 
There is  $C_0>0$ such that for all sufficiently large $j$  
$$C_0\inv\  \| H^s_{(\xi,x)}( y_j)\| \le \| H^u_{(\xi,x)}( z_j )\|\le    C_0 \ \| H^s_{(\xi,x)}( y_j)\|.$$ 
\end{lemma}
 
	\begin{figure}[h]
							\psscalebox{1 } 
							{
							\psset{unit=.7cm}
							\begin{pspicture}(1,-3.)(9,3.9)
							\def\dsz{5pt}
							\def\lwd{1.5pt}
							\def\labsizea{\small}
							\def\labsizeb{\small}
							
							\psbezier[linecolor=black, linewidth=\lwd](1.4629215,3.1114335)(1.6629214,2.4314334)(1.7429215,1.9314334)(1.8629215,1.3114334)(1.9829215,0.6914334)(1.9629215,-0.56856656)(1.8629673,-1.2977484)(1.7629215,-2.0085666)(1.6829214,-2.5885665)(1.2629215,-3.0885665)
							\psbezier[linecolor=black, linewidth=\lwd](6.4629216,3.1114335)(6.3429213,2.5318358)(6.1829214,2.0703642)(6.2029214,1.0849818)(6.2229214,0.09959931)(6.382921,-0.46366644)(6.4629216,-1.2885666)(6.5429215,-2.1418538)(6.4829216,-2.6400597)(6.2629213,-3.1085665)
							\psbezier[linecolor=darkgray,linestyle=dashed, linewidth=\lwd](9.062922,-2.0885665)(8.262921,-1.9285666)(7.862921,-1.8885666)(6.4629216,-1.2885666)(5.0629215,-0.68856657)(2.8629215,0.7114334)(1.8629215,1.3114334)(0.8629215,1.9114335)(0.84292144,1.9714334)(0.042921446,2.5714333)

							\psbezier[linecolor=black, linewidth=\lwd](0.0033587317,-1.4236704)(0.7227816,-1.3011878)(0.86325127,-1.2739145)(1.8629673,-1.2977484)(2.862683,-1.3215824)(5.373685,-1.370877)(6.462944,-1.2831012)(7.5522027,-1.1953255)(8.361149,-0.93571675)(9.040381,-0.75355273)
							
							\psbezier[linecolor=darkgray,linestyle=dashed, linewidth=\lwd]
							(4.4629677,-2.3885665)
							(3.6629667,-1.9285666)
							(3.2629667,-1.8885666)
							(1.8629673,-1.2977484)
							(1.8629673,-1.2977484)
							(1, -.8)
							(0.0033587317,-.4236704)

							\psdots[linecolor=black, dotsize=\dsz](1.8629673,-1.2977484)
							\psdots[linecolor=black, dotsize=\dsz](6.4629216,-1.2885666)
							\psdots[linecolor=black, dotsize=\dsz](1.8629215,1.3114334)
							
							\uput[45](1.8629673,-1.2977484){\labsizea$x$}
							\uput[45](6.4629216,-1.2885666){\labsizea$y_j$}
							\uput[45](1.8629215,1.3114334){\labsizea$z_j$}
							
							\uput[0](9.040381,-0.75355273){\labsizeb $W^s_\xi(x)$}
							\uput[-15](9.062922,-2.0885665){\labsizeb$W^s_\eta(y_j)$}
							\uput[-15](4.4629677,-2.3885665){\labsizeb$W^s_\eta(x)$}
							\uput[90](6.4629216,3.1114335){\labsizeb$W^u_\xi(y_j) =W^u_\eta(y_j)$ }
							\uput[90](1.4629215,3.1114335){\labsizeb$W^u_\xi(x) =W^u_\eta(x)$ }
							
							\end{pspicture}
							}
							\caption{Choice of $z_j$}\label{fig:1}
							\end{figure}
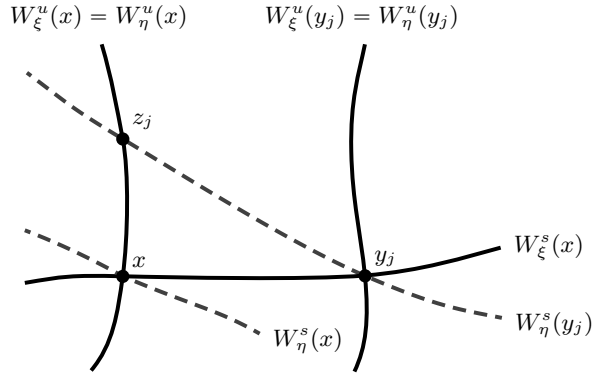

\def\C{\mathcal C}
\begin{proof}
Identify $M$ with $M_\xi$ and choose a coordinate map $\Psi \colon x\in U\subset M\to \R^2$ and $r$ small enough so that 
\begin{itemize}
\item $\Psi(x) = 0$, 
\item  $\Psi(\locstabM x \xi )$ and $\Psi(\locunstM  x \xi )$ are contained in the coordinate axes, and   
 \item $\Psi(\locstabM [r] x \eta)$  
 is contained in the interior of a closed cone $\C$ that intersects the axes only in the origin.  
\end{itemize}
By the $C^1$-continuity of the local stable manifolds on $K\cap M_\eta$, for sufficiently large $j$ $\Psi(\locstabM [r] {y_j} \eta)$ is contained in the interior of the cone $ \Psi(y_j)+\C $.  

From trigonometry there is a $\td C$ with
$$ \td C\inv d(0, \Psi(y_j)) \le 
d(0, \Psi(\locstabM [r] {y_j} \eta )\cap \Psi(\locunstM [r] {x} \eta)\le
\td C d(0, \Psi(y_j)) $$
From the bi-Lipschitz bounds on $\Psi$ and on $ H^s_{(\xi,x)}$ and $H^u_{(\xi,x)}$ restricted to local manifolds $\locstabM x \xi$ and $ \locunstM x \xi$, the estimates follow in the affine coordinates.
\end{proof}

We establish some conventions 
for the remainder.  (See Figures \ref{fig:1} and  \ref{fig:2}.)
\begin{definition} 
For $j\in \good$ satisfying Lemmas \ref{lem:y} and \ref{lem:riemGeom}  define 
$q_j := (\xi, y_j)$,
  $\ponej : = F^{\tau_{1,\epsilon}(j)}({(\xi,x)})$,
 $\qonej : = F^{\tau_{1,\epsilon}(j)}({\xi, y_j})$,
 $\ptwoj : = F^{\tau_{2,\epsilon}(j)}({\eta, x})$,
 $\qtwoj : = F^{\tau_{2,\epsilon}(j)}({\eta, y_j})$,  $\zj: =\locstabM {y_j} \eta \cap \locunstM x \xi$,  
 $\wj:= \cocycle [\eta][{\tau_{2,\epsilon}(j)}](z_j)$, and $\oj  =  F^{\tau_{2,\epsilon}(j)}({\eta, z_j}).$
\end{definition}

We  establish controls on  the growth of certain quantities.  
\begin{lemma}\label{lem:boundsonaffine}
There are $C_2>0$ and $C_3>0$ such that for all sufficiently large $j\in \good$ 
\begin{enumerate}[label={(\alph*)}, font=\normalfont]
\item \label{item5:1}
	$C_3\inv\epsilon\le \| H^u_{\ptwoj}(\wj)\| \le C_3\epsilon  $
\item \label{item5:2}
	$\displaystyle{C_2\inv\le  \left\| \restrict {D F^{\tau_{1,\epsilon}(j)}}{E^u({(\xi,x)})} \right\|\    \left\| \restrict {D F^{\tau_{2,\epsilon}(j)}}{E^u({\eta, x})} \right\|\inv  \le C_2 }$
\end{enumerate}
and if $\epsilon<\hat r/C_3$
\begin{enumerate}[resume, label={(\alph*)}, font=\normalfont]
\item \label{item5:3}	
	$\displaystyle{C_2\inv\le \left\| \restrict{ D F^{\tau_{1,\epsilon}(j)}}{E^u({\xi, y_j})}\right \|   \left \| \restrict {DF^{\tau_{2,\epsilon}(j)}}{E^u({\eta, y_j})}\right\| \inv \le C_2 }$
\end{enumerate}
\end{lemma}

\begin{proof}


Let $M_0 = \displaystyle \max_{q\in K} \left\{  \dfrac{  \|\cdot\|_{q}}{  \lyap \cdot  _{q, \eps}},\dfrac{  \lyap \cdot  _{q, \eps}}{  \|\cdot\|_{q}} \right\}$. 
Write $p = (\xi,x)$.  

We have $N_0\inv \le \| H^s_{F^{-j}(p)}(\cocycle[\xi][-j] ({y_j}))\|\le 1 $ whence we derive the sequence of bounds
\begin{align*}
M_0^{-1}N_0\inv &\le \lyap{ H^s_{F^{-j}(p)}(\cocycle[\xi][-j] ({y_j}))}_{\eps}\le M_0\\
M_0^{-1}N_0\inv\lyap{\restrict{DF^{-j}} {\Es x \xi}}_{\eps} \inv
& \le \lyap{H^s_{p}(y_j)}_{\eps}
\le M_0 \lyap{\restrict{DF^{-j}} {\Es x \xi}}_{\eps} \inv \\
M_0^{-2}N_0\inv\lyap{\restrict{DF^{-j}} {\Es x \xi}}_{\eps} \inv
 &\le \|H^s_{p}(y_j)\|
\le M_0^2 \lyap{\restrict{DF^{-j}} {\Es x \xi}}_{\eps} \inv \\
C_0\inv M_0^{-2}N_0\inv\lyap{\restrict{DF^{-j}} {\Es x \xi}}_{\eps} \inv
 &\le \|H^u_{p}(\zj)\| =\|H^u_{(\eta,x)}(\zj)\|
\le C_0M_0^2 \lyap{\restrict{DF^{-j}} {\Es x \xi}}_{\eps} \inv \\
C_0\inv M_0^{-3}N_0\inv\lyap{\restrict{DF^{-j}} {\Es x \xi}}_{\eps} \inv
& \le \lyap{H^u_{(\eta,x)}(\zj)}_{\eps}
\le C_0M_0^3 \lyap{\restrict{DF^{-j}} {\Es x \xi}}_{\eps} \inv.
\end{align*}
We thus have \begin{align*} 
C_0\inv M_0^{-3}&N_0\inv \lyap{\restrict{DF^{\tau_{2, \epsilon}(j)}}{\Eu x \eta}}_{\eps} 
\lyap{\restrict{DF^{-j}} {\Es x \xi}}_{\eps} \inv
 \\&  \le \lyap{H^u_{\ptwoj }( \cocycle [\eta][{\tau_{2,\epsilon}(j)}](z_j))}_{\eps} 
 	\le C_0M_0^3 \lyap{\restrict{DF^{\tau_{2, \epsilon}(j)}}{\Eu x \eta}}_{\eps} 
	\lyap{\restrict{DF^{-j}} {\Es x \xi}}_{\eps} \inv.
\end{align*}
from which we conclude 
$$C_0\inv M_0^{-4}N_0\inv e^{-a} \epsilon
  \le \|H^u_{\ptwoj}(\wj))\|
 \le C_0M_0^4 \epsilon.$$
 This proves \ref{item5:1} with $C_3 = C_0M_0^4N_0e^a$.

For \ref{item5:2}, we have, in the Lyapunov norms,
$$ e^{-a} \epsilon\le \lyap{\restrict{DF^{\tau_{1,\epsilon}(j)}}{\Eu x \xi}}_{\eps}  \lyap{\restrict{DF^{-j}} {\Es x \xi}}_{\eps} \inv\le \epsilon$$
and $$
e^{-a} \epsilon\le \lyap{\restrict{DF^{\tau_{2,\epsilon}(j)}}{\Eu x \eta}}_{\eps}  \lyap{\restrict{DF^{-j}} {\Es x \xi}}_{\eps} \inv\le \epsilon
$$
whence 
$$e^{-a} \le \lyap{\restrict{DF^{\tau_{1,\epsilon}(j)}}{\Eu x \xi}}_{\eps}  \lyap{\restrict{DF^{\tau_{2,\epsilon}(j)}}{\Eu x \eta}}_{\eps}\inv \le e^a.$$
Converting  to the  Riemannian norm we have
$$e^{-a}M_0^{-4} \le \|\restrict{DF^{\tau_{1,\epsilon}(j)}}{\Eu x \xi}\| \  \|\restrict{DF^{\tau_{2,\epsilon}(j)}}{\Eu x \eta}\|\inv \le e^aM_0^4.$$

To prove \ref{item5:3} we expand 
$ \left\| \restrict{ D F^{\tau_{1,\epsilon}(j)}}{\Eu{ y_j} \xi} \right\|\ \left \| \restrict {DF^{\tau_{2,\epsilon}(j)}}{\Eu{ y_j} \eta}\right\| \inv $ into the equivalent product  

\begin{align*}
\dfrac{\|\restrict{ D F^{\tau_{1,\epsilon}(j)}}{\Eu{ y_j} \xi}\| }
			{\|\restrict{ D F^{\tau_{1,\epsilon}(j)}}{\Eu{ x} \xi}\| }
\dfrac{\|\restrict{ D F^{\tau_{1,\epsilon}(j)}}{\Eu{ x} \xi}\| }
			{\|\restrict{ D F^{\tau_{2,\epsilon}(j)}}{\Eu{ x} \eta}\| }
\dfrac{\|\restrict{ D F^{\tau_{2,\epsilon}(j)}}{\Eu{ z_j} \eta }\| \inv}
			{\|\restrict{ D F^{\tau_{2,\epsilon}(j)}}{\Eu{ x} \eta}\| \inv}
\dfrac{\|\restrict{ D F^{\tau_{2,\epsilon}(j)}}{\Eu{ z_j} \eta }\| }
{ \| \restrict {DF^{\tau_{2,\epsilon}(j)}}{\Eu{ y_j} \eta}\|}.
\end{align*}

For $\epsilon>0$ such that $C_3 \epsilon<\rhere$, we have upper and lower bounds---uniform in $j\in \good$---on each ratio: 
 The bounds on the first and forth ratios  follow from Lemma \ref{lem:standardcrap}(\ref{item8:2}),  the second  from \ref{item5:2}, and the third from 
\ref{item5:1} and Lemma \ref{lem:standardcrap}(\ref{item8:1}).  
\end{proof}

\subsection{Step \step: Construction of $G_\epsilon$}
For the remainder, consider any $0<\epsilon\le \frac {1}{2 C_3} \min\{\hat r, 1\}$.  
 Then for $j\in \good$ large enough, we have $\oj \in \locunstp[1]{\ptwoj}$ and all bounds in the above lemmas hold.  
We define  $G_\epsilon\subset K $ to be the set of accumulation points of $\{\ptwoj\}_{j\in \good}$:
	$$G_\epsilon:= \bigcap_{M\to \infty}\overline{\{ \ptwoj \mid j\in \good \cap [M,\infty)\} }=  \bigcap_{M\to \infty} \overline{\{ F^{\ell}(\eta,x) \mid \ell\in \tau_{2,\epsilon}(\good  \cap [M,\infty))\}  }.$$
We note   by the construction of $\good$ that $G_\epsilon \subset K$ and is non-empty.  Furthermore, using that $(\eta,x) \in  \rec (O)$ for any $O\in \U^*$,  by Lemmas \ref{lem:jointreturns} and  \ref{lem:measureofaccumulation} and \eqref{eq:delta0} we have
$$\mu(G_\epsilon) \ge \bar d(\tau_{2,\epsilon}(\good)) \ge\delta_0.$$

We show that Lemma  \ref{lem:main} holds with $G_\epsilon$ as defined above.  
Let $$C_4:= \sup_{p\in K} \sup _{q\in \locunstp [1] p} \| D_0H^u_q\circ (H^u_p)\inv\|.$$  Then the constant $M$ in  Lemma \ref{lem:main} is given by 
	$$M:= C_3C_4C_2^2.$$  We note this is independent of $\epsilon$.  
 
\begin{lemma}
Let $p\in G_\epsilon$.  Then there is 
an affine map 
	$$\psi\colon \R\to \R$$ with 
	\begin{enumerate}
	\item $\dfrac{1}{M}\le | D\psi|\le {M}$;
	\item $\dfrac{\epsilon}{M} \le |\psi(0)|\le M \epsilon$;
	\item $\psi_*\bar \mu_{p}\simeq \bar\mu_{p}$. 
	\end{enumerate}
\end{lemma}

\begin{proof}

\def\oddlim{\lim_{j\in B \to \infty}}
We have each $\ptwoj$ and $\qtwoj$ is contained in the compact set $K$.  Let $p\in G_\epsilon$ be an accumulation point of $\{\ptwoj\}$.
We may restrict to an infinite subset $B\subset \good \subset \N_0$ such that 
$\displaystyle{\oddlim}\ptwoj = p$
and such that the sequence $\{\qtwoj\}_{j\in B}$ converges. Let 
$$q=  \displaystyle{\oddlim} \qtwoj = \displaystyle{\oddlim} \oj.$$ 
Note that $q \in K$ 
 and by Lemma \ref{lem:boundsonaffine}\ref{item5:1},  $q\in \locunstp[1] {p}$.  See Figure \ref{fig:2}.

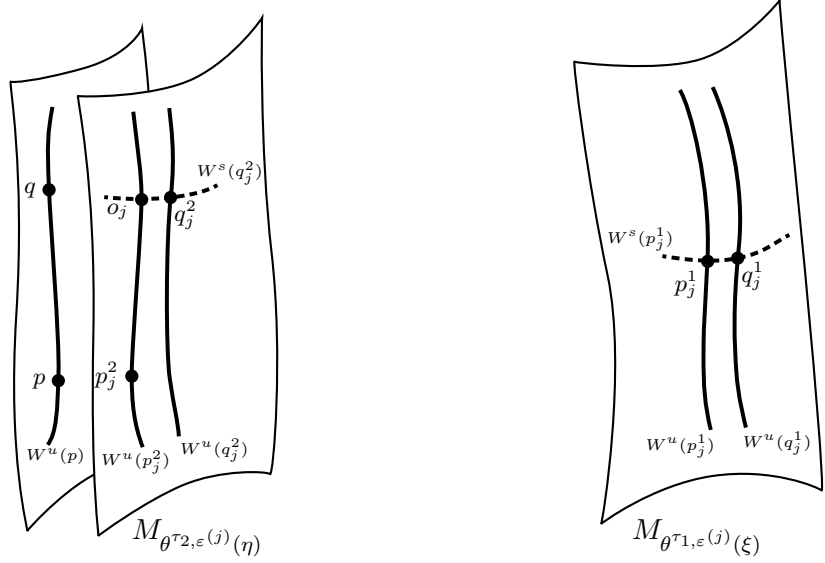
\begin{figure}[h]
\def\dsz{5pt}
\def\lwd{1.5pt}
\def\lwda{.8pt}
\def\framecolor{black}
\def\labsizea{\small}
\def\labsizeb{\large}
\def\labsizec{\tiny}
\centering
\begin{minipage}{.5\textwidth}
  \centering
\psscalebox{1.0 1.0} 
{
\psset{unit=.8cm}
\begin{pspicture}(0,-4.4)(5.082738,4.3)

\psbezier[linecolor=\framecolor, linewidth=\lwda](0.296937,-4.0879235)(0.27284434,-4.0986233)(1.084102,-3.175938)(1.7426844,-3.0102384)(2.4012668,-2.8445392)(2.4577107,-2.4845803)(2.5130274,-2.5284376)(2.5683439,-2.572295)(2.2141259,-1.9009148)(2.3439364,0.33267576)(2.473747,2.5662663)(2.1978126,4.2753954)(2.206798,4.276514)(2.2157838,4.2776327)(2.0073109,3.8254263)(1.2237635,3.6071656)(0.4402161,3.388905)(0.055362187,3.370336)(0.022785649,3.3864863)(-0.009790889,3.4026365)(0.26676044,1.9354573)(0.11684643,-0.4872773)(-0.033067577,-2.910012)(0.32102966,-4.0772233)(0.296937,-4.0879235)
\psbezier[linecolor=\framecolor, fillstyle=solid, linewidth=\lwda](1.4769369,-4.1679235)(1.4528444,-4.179189)(2.244102,-3.5920427)(2.8826845,-3.3239274)(3.5212667,-3.0558121)(4.277711,-3.0917156)(4.3330274,-3.1378915)(4.388344,-3.1840675)(4.5567346,-1.1729321)(4.358719,0.5452335)(4.1607037,2.2633991)(4.1935735,4.451423)(4.202559,4.452601)(4.2115445,4.4537787)(3.742311,3.777833)(3.0812635,3.4466896)(2.420216,3.115546)(1.1753622,3.1328232)(1.1427857,3.1498275)(1.1102091,3.1668315)(1.4467604,1.769488)(1.3968464,-0.8234444)(1.3469324,-3.4163768)(1.5010296,-4.1566577)(1.4769369,-4.1679235)

\psbezier[linecolor=black, linewidth=\lwd](0.71798676,2.9680867)(0.635195,2.5242352)(0.6284979,2.3625515)(0.65798676,1.6386243)(0.68747556,0.9146972)(0.835195,-0.8397075)(0.8179867,-1.4634545)(0.8007785,-2.0872016)(0.7874756,-2.4492455)(0.6379868,-2.6519134)
\psbezier[linecolor=black, linewidth=\lwd](2.060209,2.8880868)(2.150504,2.097337)(2.1875856,1.9387298)(2.1979868,1.5037915)(2.2083879,1.068853)(2.0401337,-0.93380386)(2.0350237,-1.4490945)(2.0299137,-1.9643852)(2.0780175,-2.3255093)(2.2179868,-2.6919134)
\psbezier[linecolor=black, linewidth=\lwd](2.6579866,2.9480867)(2.6847079,2.6618645)(2.7645957,2.1874063)(2.6839328,1.4789958)(2.6032696,0.7705853)(2.5976808,-1.0689656)(2.6666355,-1.5470911)(2.73559,-2.0252168)(2.8027291,-2.2649877)(2.8179867,-2.5119133)
\psbezier[linecolor=black, linestyle=dashed,  dash=3pt 2pt, linewidth=\lwd](1.5779867,1.4680867)(2.3804257,1.3805867)(3.1599379,1.4805866)(3.4579868,1.6680866)

\psdots[dotsize=\dsz]
(0.817,-1.5869133)
(0.66048676,1.5880867)
(2.68,1.4630867)
(2.1979868,1.4255867)
(2.0354867,-1.5119133)

\uput[180](0.817,-1.5869133){\labsizea$p$}
\uput[180](0.66048676,1.5880867){\labsizea$q$}
\uput[210]	(2.1979868,1.4255867)	{\labsizea$\oj$}
\uput{2pt}[-45]	(2.68,1.4630867){\labsizea$\qtwoj$}
\uput[180]	(2.0354867,-1.5119133)	{\labsizea$\ptwoj$}

\uput[0](1.821465,-4.2){\labsizeb $M_{\theta^{\tau_{2,\epsilon}(j)}(\eta)}$}

\uput{1pt}[70](3.4479868,1.6380867){\labsizec$W^s (\qtwoj)$}
\uput{1pt}[315](2.8104868,-2.5244133){\labsizec $W^u(\qtwoj)$}
\uput{1pt}[260](2.2104867,-2.6994133){\labsizec $W^u(\ptwoj)$}
\uput{1pt}[283](0.64798677,-2.6619134){\labsizec $W^u(p)$}
\end{pspicture}
}

\end{minipage}%
\begin{minipage}{.5\textwidth}
  \centering

\psscalebox{1.0 1.0} 
{
\psset{unit=.8cm}
\begin{pspicture}(8,-4.4)(14,4.3)
%
\psbezier[linecolor=\framecolor, linewidth=\lwda](10.39027,-3.9679234)(10.366178,-3.9786234)(11.2441025,-3.3159378)(12.242684,-3.1702385)(13.241267,-3.0245392)(13.997711,-3.3645804)(14.053027,-3.4084377)(14.108344,-3.452295)(13.914125,-1.1009147)(13.693936,1.1051757)(13.473747,3.3112662)(13.347813,4.7378955)(13.356798,4.739014)(13.365784,4.740133)(12.742311,3.9779263)(12.038764,3.7596657)(11.3352165,3.541405)(9.980362,3.702836)(9.947785,3.7189863)(9.915209,3.7351365)(10.16676,1.7154573)(10.496846,0.17272268)(10.826932,-1.3700119)(10.414363,-3.9572232)(10.39027,-3.9679234)
%
\psbezier[linecolor=black, linewidth=\lwd](11.697987,3.2480867)(11.868731,2.9921746)(12.217183,1.6336619)(12.158442,0.43464482)(12.0997,-0.7643723)(12.094607,-0.6210142)(12.092184,-1.1526922)(12.089761,-1.6843702)(12.111489,-1.9847519)(12.212987,-2.4044132)
\psbezier[linecolor=black, linewidth=\lwd](12.244944,3.2967823)(12.505839,2.7131937)(12.782093,1.6692182)(12.679001,0.6361691)(12.57591,-0.39688006)(12.606564,-0.9682708)(12.637335,-1.4267875)(12.668106,-1.8853041)(12.76347,-2.210006)(12.797117,-2.3649569)
\psbezier[linecolor=black, linestyle=dashed, dash=3pt 2pt, linewidth=\lwd](13.517986,0.8400867)(13.270812,0.7100026)(13.121784,0.56392753)(12.782887,0.48012587)(12.44399,0.39632422)(12.106798,0.3762636)(11.409291,0.4865015)
%
\psdots[dotsize=\dsz]
(12.16,0.4)
(12.66,0.45330408)
%
\uput[225]	(12.16,0.4) {\labsizea$\ponej$}
\uput{2pt}[-45](12.66,0.45330408) {\labsizea$\qonej$}
%
%
\uput[0](10.686683,-4.2){\labsizeb $M_{\theta^{\tau_{1,\epsilon}(j)}(\xi)}$}
%
\uput{1pt}[310](12.797987,-2.3869133){\labsizec $W^u(\qonej)$}
\uput{2pt}[120](11.410487,0.4880867){\labsizec $W^s (\ponej)$}
\uput{1pt}[230](12.222987,-2.4369133){\labsizec $W^u(\ponej)$}
\end{pspicture}
}

\end{minipage}
\caption{Proof of Lemma \ref{lem:main}}\label{fig:2}
\end{figure}

Fix $\gamma := d(H^u_{q}\circ (H^u_{p})\inv(t))/dt(0)$ and let $v:= \I_p\circ H^u_{p} (q)$. 
Note that by Lemma \ref{lem:boundsonaffine}\ref{item5:1}, we have $C_3\inv \epsilon \le |v| \le C_3\epsilon.$
Define the map $\Phi\colon \R\to \R$  by $$\Phi\colon t \mapsto \gamma (t- v).$$  By  construction,  we have
\begin{equation}\label{eq:half}\Phi_*\bar \mu_{p} \simeq \bar \mu_{q}.\end{equation}
It remains to relate  the measures $\bar \mu_{q}$  and $\bar \mu_{p}$.  

\def\oddlim{\lim_{j\in B \to \infty}}

Given $\alpha\in \R$, write $\lambda_{\alpha}\colon \R\to \R$ for the 
linear map  $\lambda_\alpha\colon x \mapsto \alpha x$.
Define $\alpha_j, \beta_j\in \R$ so that
	$$\lambda_{\alpha_j} = \I_{\ponej}\circ  \restrict {D F^{\tau_{1,\epsilon}(j)}}{E^u({(\xi,x)})} \circ D\hol_{\eta, \xi}\circ \left(\restrict {D F^{ \tau_{2,\epsilon}(j)}}{E^u({\eta, x})}\right)^{-1}  \circ \I_{\ptwoj}\inv,$$
	$$\lambda_{\beta_j }=  \I_{\qonej} \circ\restrict{ D F^{\tau_{1,\epsilon}(j)}}{E^u({\xi, y_j})} \circ D\hol_{\eta, \xi}\circ \left( \restrict {DF^{\tau_{2,\epsilon}(j)}}{E^u({\eta, y_j})} \right)^{-1}\circ  \I_{\qtwoj}\inv .$$
As before, $\hol_{\eta, \xi}$ denotes the trivial identification between $M_\eta$ and $M_\xi$.  
From  Lemma \ref{lem:boundsonaffine} we 
have 
$$|\alpha_j |\in [C_2\inv, C_2],\quad |\beta_j |\in [C_2\inv, C_2]$$
hence we 
may further restrict the set $B\subset \good \subset \N_0$ so that the limits 
$$\displaystyle{\oddlim} \alpha_j = \alpha, \quad \quad \displaystyle{\oddlim}\beta_j = \beta$$ are defined.

We claim that $(\lambda_\alpha)_*\bar\mu_p\simeq (\lambda_\beta)_*\bar\mu_q.$
 Indeed, we have 
$$(\lambda_{\alpha_j})_* \bar\mu_{\ptwoj} \simeq \bar\mu_{\ponej}, \quad \quad (\lambda_{\beta_j})_* \bar\mu_{\qtwoj} \simeq \bar\mu_{\qonej}.$$
We introduce normalization factors
$$c_j := \bar\mu_{\ptwoj}([-\alpha_j\inv, \alpha_j\inv]) \inv, \quad \quad d_j := \bar\mu_{\qtwoj}([-\beta_j\inv, \beta_j\inv])\inv$$ and 
$$c := \bar\mu_{p}([-\alpha\inv, \alpha\inv])\inv, \quad \quad d := \bar\mu_{q}([-\beta\inv, \beta\inv])\inv.$$
We recall that each $\bar \mu_q$ has no atoms; hence intervals are continuity sets for each $\bar \mu_q$ and thus $c_j\to c$ and $d_j\to d$.  
Let $f$ be a continuous, compactly supported function $f\colon \R \to \R$.  We note that $q\mapsto \bar\mu_q(f)$ is uniformly continuous on $K$ and that
$$\left | (\lambda_{\alpha})_*\bar\mu_{q}(f) 
			-(\lambda_{\alpha_j})_*\bar\mu_{q}(f)  \right| 
=  \left | \int f(\alpha t) - f( {\alpha_j} t)  \ d\bar\mu_{q}(t) \right| $$
approaches zero uniformly in $q$ as $j\in B\to \infty$.
Thus for any $\kappa>0$ and for all sufficiently large $j\in B$ we have 
\begin{itemize}
\item $ | {c}(\lambda_\alpha)_*\bar\mu_{p}(f)-  {c}(\lambda_\alpha)_*\bar\mu_{\ptwoj}(f)| \le \kappa$,
\item $ |  {c}(\lambda_\alpha)_*\bar\mu_{\ptwoj}(f)
			- {c_j}(\lambda_{\alpha_j})_*\bar\mu_{\ptwoj}(f) | \le \kappa$,
			
\item $ | {d}(\lambda_\beta)_*\bar\mu_{q}(f)-  {d}(\lambda_\beta)_*\bar\mu_{\qtwoj}(f)| \le \kappa$,
\item $ |  {d}(\lambda_\beta)_*\bar\mu_{\qtwoj}(f)
			- {d_j}(\lambda_{\beta_j})_*\bar\mu_{\qtwoj}(f) | \le \kappa$,
\item  $|  \bar\mu_{\ponej} (f)- \bar\mu_{\qonej}(f)|\le \kappa$. 
 \end{itemize}
 
 Since
$$ {c_j}(\lambda_{\alpha_j})_*\bar\mu_{\ptwoj}(f)  =  \bar\mu_{\ponej}(f),\quad \quad  {d_j}(\lambda_{\beta_j})_*\bar\mu_{\qtwoj}(f) =  \bar\mu_{\qonej}(f)$$
we conclude 
$c(\lambda_\alpha)_*\bar\mu_{p} = d(\lambda_\beta)_*\bar\mu_{q},$
or $$\bar \mu_{q}\simeq  (\lambda_{\alpha/\beta})_*\bar\mu_p.$$
Combining the above  with \eqref{eq:half}, it follows that map 
$$\psi=(\lambda_{\alpha/\beta})\inv\circ \Phi\colon t \mapsto \dfrac{\beta \gamma}{\alpha} (t-v).$$
satisfies the conclusion of the Lemma.  
\end{proof}

This completes the proof of Lemma \ref{lem:main}.

\section{Proof of Theorem \ref{thm:3}}
We end with the proof of Theorem \ref{thm:3}.   
\begin{proof}
Let $\munaught$ be as in Theorem \ref{thm:3}, and assume that $h_\munaught(\MP^+(M, \nunaught))>0$ and the stable distribution $E^s_\omega(x)$ is non-random.  It follows from Theorem \ref{thm:1} that $\munaught$ is SRB.  
\def\B{\mathcal B}
Let $F\colon \Sigma \times M\to \Sigma \times M$ be the canonical skew product  constructed in Section \ref{sec:skewRDS} and let $\mu$ be the $F$-invariant measure defined by Proposition \ref{prop:mudef}. Then   
the conditional measures of $\mu$ along \ae unstable manifold $\unst \xi x$ for the skew product $F$ are absolutely continuous.  
Define the \emph{ergodic basin}  $B\subset \Sigma\times M$  of $\mu$ to be the set of $(\xi,x) \in X$ such that 
 $$ \lim_{n\to \infty}\tfrac 1 N \sum _{n = 0}^{N-1} \phi(\cocycle  [\xi] [n](x))= \int \phi \ d \munaught$$
 for all $\phi\colon M \to \R$ continuous. 
By the point-wise ergodic theorem and the separability of $C^0(M)$, we have $\mu(B) =1.$
Furthermore, for points $(\xi, x)\in B$  whose fiber-wise stable manifold $\stab x \xi $ is defined 
we have $$\stab x \xi \in B .$$

We have  the following ``transverse'' absolute continuity property.   
Given a typical $\xi\in \Sigma$ and a collection of  
fiber-wise local stable manifolds $\mathcal S:= \{\locstabM x \xi \}_{x\in Q}$ with ``bounded geometry'' consider two manifolds $T_1$ and $T_2$ everywhere transverse to the collection $\mathcal S$.  Define the holonomy map from $T_1$ to $T_2$ by ``sliding along'' elements of $\mathcal S$.  Such holonomy maps 
were shown by Pesin to be absolutely continuous in the deterministic volume preserving setting \cite{MR0458490}.  
For fiber-wise stable manifolds associated to skew products satisfying \eqref{eq:IC2}, such holonomy maps are also known to be absolutely continuous.  See \cite[(4.2)]{MR968818} for further details and references and  to proofs.

The above absolute continuity property implies that  if $\munaught$ is SRB and if $A\subset\Sigma\times M$ is any set with $\mu(A)>0$ then for $\nu^\Z$-\ae $\xi$
$$\bigcup_{(\xi,x) \in A\cap M_\xi}\stabM x \xi \subset M_\xi$$
has positive Lebesgue measure.   
It follows that for the ergodic basis $B$, $$(\nunaught^\Z\times m)( B)>0.$$
We note that if $\eta\in \Sigmalocs(\xi)$ then $$\hol_{\eta,\xi}(B\cap M_\eta) = B\cap M_\xi$$
since $\cocycle[\xi][n]= \cocycle[\eta][n] $ for $n\ge 0$.  
Define $\hat B$ to be the ergodic basin of $\nu^\N\times \munaught$ for the skew product $\hat F\colon \Sigma_+\times M$; that is $(\omega, x)\in \hat B$ if 
$$ \lim_{n\to \infty}\tfrac 1 N \sum _{n = 0}^{N-1} \phi(\cocycle  [\omega] [n](x))= \int \phi \ d \munaught$$
 for all $\phi\colon M \to \R$ continuous. 
 We have that $\hat B$ is the image of $B$ under the natural projection $\Sigma\times M \to \Sigma _+ \times M$ whence 
 $ {(\nunaught^\N\times m)(\hat  B)} >0$.

Define a measure $$\hat m= \tfrac 1 {(\nunaught^\N\times m)(\hat  B)} \restrict {(\nunaught^\N\times m)}{ \hat B}$$ on $\Sigma_+\times M$.  Since both the set $\hat  B$ and the measure $\nunaught^\N\times m$ are $\hat F$-invariant (recall that $m$ is $\nunaught$-\as invariant) the measure $\hat m$ is $\hat F$-invariant.  
  Furthermore, for $\hat m$-\ae $(\omega, x)$ and any continuous $\phi\colon M \to \R,$ the Birkhoff sums satisfy $$\lim_{n\to \infty}\frac 1 N \sum _{n = 0}^{N-1} \phi(\cocycle  [\omega] [n] (x)) = \int \phi \ d \munaught$$ which  implies that $\hat m$ is ergodic for $F$ and, in particular, is an ergodic component of $\nunaught^\Z\times m$.  This implies (see e.g.\ \cite[Proposition I.2.1]{MR884892}) that $\hat m $ is of the form $\hat m= \nunaught^\Z\times m_0$ for $m_0$ an ergodic component of $m$ for $\MP^+(M, \nu)$.  

Then, for any continuous function $\phi\colon M\to \R$,  $\nunaught^\N$-\ae $\omega\in \Sigma_+$, and $m_0$-a.e. $x\in M$, we have 
$$\lim_{n\to \infty}\frac 1 N \sum _{n = 0}^{N-1} \phi(\cocycle [\omega]  [n] (x) )= \int \phi \ d \munaught.$$
Furthermore, since  $\nunaught \times m_0$ is invariant and ergodic  for $\hat F$, for $\nunaught^\N$-\ae $\omega\in \Sigma_+$  and $m_0$-a.e. $x\in M$   we also have that 
$$\lim_{n\to \infty}\frac 1 N \sum _{n = 0}^{N-1} \phi(\cocycle [\omega]  [n] (x)) = \int \phi \ d m_0.$$
In particular, $ \int \phi \ d \munaught =  \int \phi \ d m_0$ for all $\phi\colon M\to \R$, whence $\munaught= m_0$.  
\end{proof}

\bibliographystyle{../../AWBmath}

\bibliography{../../bibliography}

\end{document}